\documentclass[11pt]{amsart}
\usepackage{amsmath,amssymb}
\usepackage[alphabetic]{amsrefs}
\usepackage{mathrsfs}
\usepackage{graphicx,color}
\usepackage{wrapfig}
\usepackage{calc}
\usepackage{extarrows}
\usepackage{amsthm}
\usepackage{latexsym}
\usepackage{slashed}
\usepackage{framed}
\usepackage{ascmac}
\usepackage[all]{xy}
\usepackage{hyperref}
\usepackage[mathscr]{eucal}

\allowdisplaybreaks[3]

\BibSpec{book}{%
    +{}  {\PrintPrimary}                {transition}
    +{.} { \textit}                     {title}
    +{.} { }                            {part}
    +{:} { \textit}                     {subtitle}
    +{,} { \PrintEdition}               {edition}
    +{}  { \PrintEditorsB}              {editor}
    +{,} { \PrintTranslatorsC}          {translator}
    +{,} { \PrintContributions}         {contribution}
    +{,} { }                            {series}
    +{,} { \voltext}                    {volume}
    +{,} { }                            {publisher}
    +{,} { }                            {organization}
    +{,} { }                            {address}
    +{,} { }                            {status}
    +{,} { \PrintDOI}                   {doi}
    +{,} { \PrintISBNs}                 {isbn}
    +{}  { \parenthesize}               {language}
    +{}  { \PrintTranslation}           {translation}
    +{;} { \PrintReprint}               {reprint}
    +{,} { \PrintDate}                  {date}
    +{.} { }                            {note}
    +{.} {}                             {transition}
}
\BibSpec{article}{%
    +{}  {\PrintAuthors}                {author}
    +{,} { \textit}                     {title}
    +{.} { }                            {part}
    +{:} { \textit}                     {subtitle}
    +{,} { \PrintContributions}         {contribution}
    +{.} { \PrintPartials}              {partial}
    +{,} { }                            {journal}
    +{}  { \textbf}                     {volume}
    +{}  { \PrintDatePV}                {date}
    +{,} { \issuetext}                  {number}
    +{,} { \eprintpages}                {pages}
    +{,} { }                            {status}
    +{,} { \PrintDOI}                   {doi}
    +{}  { \parenthesize}               {language}
    +{}  { \PrintTranslation}           {translation}
    +{;} { \PrintReprint}               {reprint}
    +{.} { }                            {note}
    +{,} { \eprint}                     {eprint}
    +{.} {}                             {transition}
}
\BibSpec{collection.article}{%
    +{}  {\PrintAuthors}                {author}
    +{,} { \textit}                     {title}
    +{.} { }                            {part}
    +{:} { \textit}                     {subtitle}
    +{,} { \PrintContributions}         {contribution}
    +{,} { \PrintConference}            {conference}
    +{}  {\PrintBook}                   {book}
    +{,} { }                            {booktitle}
    +{,} { \PrintDateB}                 {date}
    +{,} { pp.~}                        {pages}
    +{,} { }                            {publisher}
    +{,} { }                            {organization}
    +{,} { }                            {address}
    +{,} { }                            {status}
    +{,} { \PrintDOI}                   {doi}
    +{,} { \eprint}        {eprint}
    +{}  { \parenthesize}               {language}
    +{}  { \PrintTranslation}           {translation}
    +{;} { \PrintReprint}               {reprint}
    +{.} { }                            {note}
    +{.} {}                             {transition}
}

\BibSpec{misc}{
  +{}{\PrintAuthors}  {author}
  +{,}{ \textit}      {title}
  +{,}{ }             {note}
  +{,}{ }             {date}
  +{.}{}              {transition}
}

\makeatletter
\@addtoreset{equation}{section}
\makeatother

\DeclareMathAlphabet{\mathfrak}{U}{euf}{m}{n}
\SetMathAlphabet{\mathfrak}{bold}{U}{euf}{b}{n}

\newcommand{\ma}[1]{\begin{align*} #1 \end{align*}}
\newcommand{\maa}[1]{\begin{align} #1 \end{align}}

\theoremstyle{definition}
\newtheorem{defn}[equation]{Definition}
\theoremstyle{plain}
\newtheorem{thm}[equation]{Theorem}
\newtheorem{prp}[equation]{Proposition}
\newtheorem{lem}[equation]{Lemma}
\newtheorem{cor}[equation]{Corollary}

\theoremstyle{remark}
\newtheorem{remk}[equation]{Remark}
\newtheorem{exmp}[equation]{Example}
\newtheorem*{remk*}{Remark}
\newtheorem{blank}[equation]{}

\newcommand{\real}{\mathbb R}
\newcommand{\comp}{\mathbb C}
\newcommand{\zahl}{\mathbb Z}

\newcommand{\nat}{\mathbb N}

\DeclareMathOperator{\Hom}{Hom}

\DeclareMathOperator{\Tor}{Tor}

\newcommand{\id}{\mathrm{id}}
\DeclareMathOperator{\rank}{rank}
\DeclareMathOperator{\Ker}{Ker}

\newcommand{\ebk}[1]{\left<  #1 \right>}

\newcommand{\ssbk}[1]{\left\| #1 \right\|}

\newcommand{\xra}{\xrightarrow}

\newcommand{\loc}{\mathrm{loc}}

\DeclareMathOperator{\image}{Im}
\renewcommand{\Im}{\image}
\DeclareMathOperator{\Ind}{Ind}
\DeclareMathOperator{\Res}{Res}
\newcommand{\cone}{\mathcal{C}}
\DeclareMathOperator*{\bigast}{\text{\scalebox{1.5}{\raisebox{-0.2ex}{$\ast$}}}}

\newcommand{\GCsep}[1]{#1\mathchar`-\mathfrak{C}^*\mathfrak{sep}}
\newcommand{\ev}{\mathrm{ev}}

\newcommand{\Bad}{\mathsf{Bad}}
\renewcommand{\blank}{\text{\textvisiblespace}}
\DeclareMathOperator{\Tel}{Tel}
\newcommand{\Pro}{\mathrm{Pro}}
\newcommand{\Tcat}{\mathfrak{T}}

\newcommand{\Cat}{\mathfrak{C}}
\newcommand{\Ncat}{\mathfrak{N}}
\newcommand{\Icat}{\mathfrak{I}}

\newcommand{\J}{\mathfrak{J}}
\newcommand{\K}{\mathrm{K}}
\newcommand{\KK}{\mathrm{KK}}
\newcommand{\RKK}{\mathrm{RKK}}
\newcommand{\Kas}{\mathfrak{KK}}
\newcommand{\Cst}{\mathrm{C}^*}
\newcommand{\Kop}{\mathbb{K}}

\newcommand{\Lop}{\mathbb{L}}
\newcommand{\FC}{\mathcal{FC}}
\newcommand{\FI}{\mathcal{FI}}
\newcommand{\CI}{\mathcal{CI}}
\newcommand{\CC}{\mathcal{CC}}
\newcommand{\CNI}{\mathcal{C}_{N}\mathcal{I}}
\newcommand{\CNC}{\mathcal{C}_{N}\mathcal{C}}
\newcommand{\Ab}{\mathfrak{Ab}}
\newcommand{\F}{\mathcal{F}}
\newcommand{\Hilb}{\mathscr{H}}
\newcommand{\hoprojlim}{\mathrm{ho}\mathchar`-\varprojlim}
\newcommand{\hoinjlim}{\mathrm{ho}\mathchar`-\varinjlim}
\DeclareMathOperator{\Obj}{Obj}

\DeclareMathOperator{\diag}{diag}

\DeclareMathOperator{\Ad}{Ad}

\newcommand{\Z}{\mathcal{Z}}
\label{thm:Rokinj}

\newcommand{\myhat}{
\scalebox{1.1}[1.1]{\text{\hspace{-0.1em}\raisebox{-0.9ex}[0.8ex][0.0ex]{\textasciicircum}\hspace{-0.1em}}}
}
\newcommand{\FIN}{\mathcal{E}}
\newcommand{\FINe}{\mathcal{E}_{0}}
\newcommand{\FZ}{\mathcal{EZ}}
\newcommand{\G}{\mathcal{G}}
\newcommand{\cyc}{\mathrm{cyc}}
\newcommand{\CZ}{\mathcal{CZ}}
\newcommand{\CZC}{\mathcal{CZC}}
\newcommand{\CZI}{\mathcal{CZI}}

\title[A categorical perspective on the Atiyah-Segal theorem]{A categorical perspective on the Atiyah-Segal completion theorem in $\KK$-theory}
\date{November 17, 2015.}

\author[Y. Arano]{Yuki Arano}
\address{Graduate School of Mathematical Science, The University of Tokyo, 3-8-1 Komaba, Meguro-ku, Tokyo 153-8914, Japan}
\email{arano@ms.u-tokyo.ac.jp}

\author[Y. Kubota]{Yosuke Kubota}
\address{Graduate School of Mathematical Science, The University of Tokyo, 3-8-1 Komaba, Meguro-ku, Tokyo 153-8914, Japan}
\email{ykubota@ms.u-tokyo.ac.jp}


\begin{document}
\maketitle
\begin{abstract}
We investigate the homological ideal $\J _G^H$, the kernel of the restriction functors in compact Lie group equivariant Kasparov categories. Applying the relative homological algebra developed by Meyer and Nest, we relate the Atiyah-Segal completion theorem with the comparison of $\J_G^H$ with the augmentation ideal of the representation ring. 

In relation to it, we study on the Atiyah-Segal completion theorem for groupoid equivariant $\KK$-theory, McClure's restriction map theorem and permanence property of the Baum-Connes conjecture under extensions of groups.
\end{abstract}

\tableofcontents

\section{Introduction}
Equivariant $\KK$-theory is one of the main subjects in the noncommutative topology, which deals with topological properties of $\Cst$-algebras. The main subject of this paper is the homological ideal
$$\J _G^H(A,B):=\Ker (\Res _G^H : \KK^G(A,B) \to \KK^H(A,B))$$
of the Kasparov category $\Kas^G$, whose objects are separable $G$-$\Cst$-algebras, morphisms are equivariant $\KK$-groups and composition is given by the Kasparov product. 

In \cite{MR2193334}, Meyer and Nest introduced a new approach to study the homological algebra of the Kasparov category. They observed that the Kasparov category has a canonical structure of the triangulated category. Moreover, they applied the Verdier localization for $\Kas ^G$ in order to give a categorical formulation of the Baum-Connes assembly map. Actually they prove that an analogue of the simplicial approximation in the Kasparov category is naturally isomorphic to the assembly map. Their argument is refined in \cite{MR2563811} in terms of relative homological algebra of the projective class developed by Christensen~\cite{MR1626856}. Moreover it is proved that the ABC spectral sequence (a generalization of Adams spectral sequence in relative homological algebra) for the functor $\K _*(G \ltimes \blank)$ and an object $A$ converges to the domain of the assembly map.

These results are essentially based on the fact that the induction functor $\Ind _{H}^G$ is the left adjoint of the restriction functor $\Res _G^H$ when $H \leq G$ is an open subgroup. On the other hand, it is also known that when $H \leq G$ is a cocompact subgroup, $\Ind _H^G$ is the right adjoint of $\Res _G^H$. This relation enables us to apply the homological algebra of injective class for $\KK$-theory. It should be noted that the category of separable $G$-$\Cst$-algebras is not closed under countable direct product although the fact that $\Kas ^G$ have countable direct sums plays an essential role in \citelist{\cite{MR2193334}\cite{MR2681710}\cite{MR2563811}}. Therefore, we replace the category $\GCsep{G}$ of separable $G$-$\Cst$-algebras with its (countable) pro-category. Actually, the category $\mathop{\mathrm{Pro}}_\nat \GCsep{G}$ is naturally equivalent to the category $\sigma\GCsep{G}$ of $\sigma$-$G$-$\Cst$-algebras, which is dealt with by Phillips in his study of the Atiyah-Segal completion theorem. Fortunately, $\KK$-theory for (non-equivariant) $\sigma$-$\Cst$-algebras are investigated by Bonkat~\cite{Bonkat}. We check that his definition is generalized for equivariant $\KK$-theory and obtain the following theorem.

\theoremstyle{plain}
\newtheorem*{intro1}{Theorem \ref{thm:triang} and Theorem \ref{thm:decomp}}
\begin{intro1}
For a compact group $G$, the equivariant Kasparov category $\sigma \Kas^G$ of $\sigma$-$G$-$\Cst$-algebras has a structure of the triangulated category. Moreover, for a family $\F$ of $G$, the pair of thick subcategories $(\FC,\ebk{\FI}^{\loc})$ is complementary. Here $\FC$ is the full subcategory of $\F$-contractible objects and $\FI$ is the class of $\F$-induced objects (see Definition \ref{def:FindFres}).
\end{intro1}

Next, we observe that this semi-orthogonal decomposition is related to a classical idea in equivariant $\K$-theory so called the Atiyah-Segal completion.
In the theory of equivariant cohomology, there is a canonical way to construct an equivariant general cohomology theory from a non-equivariant cohomology theory. Actually, for a compact Lie group $G$ and a $G$-$CW$-complex $X$, the general cohomology group of the new space given by the Borel construction $X \times _G EG$ is regarded as the equivariant version of the given cohomology group of $X$. On the other hand, equivariant $\K$-theory is defined in terms of equivariant vector bundles by Atiyah and Segal in \cite{MR0236951} \cite{MR0234452}. This group has a structure of modules over the representation ring $R(G)$ and hence is related to the representation theory of compact Lie groups. In 1969, Atiyah and Segal discovered a beautiful relation between them \cite{MR0259946}. When the equivariant $\K$-group $\K_G^*(X)$ of a compact $G$-space is finitely generated as an $R(G)$-module, then the completion of the equivariant $\K$-group by the augmentation ideal is actually isomorphic to the (representable) $\K$-group pf the Borel construction of $X$. 

This theorem is generalized in \cite{MR935523} for families of subgroups. The completion of $\K_G^*(X)$ by the family of ideals $I_G^H$ ($H \in \F$) is isomorphic to the equivariant $\K$-group $\K_G(X \times E_\F G)$ where $E_\F G$ is the universal $\F$-free $G$-space. On the other hand, Phillips~\cite{MR1050491} generalizes it for $\K$-theory of $\Cst$-algebras. In order to formulate the statement, he generalizes operator $\K$-theory for $\sigma$-$\Cst$-algebras in \cite{MR1050490}. Actually, this contains the Atiyah-Segal completion theorem for twisted $\K$-theory because the twisted equivariant $\K$-group is isomorphic to the $\K$-group of certain $\Cst$-algebra bundles with (twisted) group actions.

The Atiyah-Segal completion theorem is generalized for equivariant $\KK$-theory by Uuye~\cite{MR2887199}. Here he assumes that $\KK^H_*(A,B)$ are finitely generated for all subgroups $H$ of $G$ in order to regard the correspondence $X \mapsto \KK^G(A , B \otimes C(X))$ as an equivariant cohomology theory of finite type. We prove the categorical counterpart of the Atiyah-Segal completion theorem under weaker assumptions. 
\newtheorem*{intro2}{Theorem \ref{thm:IG2}}
\begin{intro2}
Let $G$ be a compact Lie group and let $A$, $B$ be $\sigma$-$\Cst$-algebras such that $\KK^G_*(A,B)$ are finitely generated for $\ast=0,1$. Then the filtrations $(\J_G^\F)^*(A,B)$ and $(I_G^\F)^* \KK^G(A,B)$ are equivalent.
\end{intro2}

Applying Theorem \ref{thm:IG2} for the relative homological algebra of the injective class, we obtain the following generalization of the Atiyah-Segal completion theorem. 
\newtheorem*{intro3}{Theorem \ref{thm:AS}}
\begin{intro3}
When $\KK^G(A,B)$ is a finitely generated $R(G)$-module, the following $R(G)$-modules are canonically isomorphic.
$$\KK^G(A,B)^{{\myhat}}_{I_G^\F } \cong \KK^G(A,\tilde{B}) \cong \RKK^G(E_\F G; A,B) \cong \sigma\Kas^G/\FC(A,B)$$
\end{intro3}
The Atiyah-Segal completion theorem for proper actions and groupoids are studied in \cite{MR1838997} and \cite{MR2955971}. We generalize Theorem \ref{thm:AS} for groupoid equivariant $\KK$-theory (Theorem \ref{thm:groupoid}) and equivariant $\KK$-theory for proper $G$-$\Cst$-algebras (Theorem \ref{thm:proper}) under certain assumptions.

Note that in some special cases we need not to assume that $\KK^G_*(A,B)$ are finitely generated. In particular, we obtain the following.

\newtheorem*{intro4}{Corollary \ref{cor:cyc}}
\begin{intro4}
Let $\Z$ be the family generated by all cyclic subgroups of $G$. Then, there is $n>0$ such that $(\J_G^\Z)^n=0$. 
\end{intro4}

It immediately follows from Corollary \ref{cor:cyc} that if $\Res _G^H A$ is $\KK^H$-contractible for any cyclic subgroup $H$ of $G$, then $A$ is $\KK^G$-contractible. This is a variation of McClure's restriction map theorem~\cite{MR862427} which is generalized by Uuye~\cite{MR2887199} for equivariant $\KK$-theory. We also revisit these theorems from categorical viewpoint and generalize Theorem 0.1 of \cite{MR2887199} (Theorem \ref{thm:rest}).

Next we apply Corollary \ref{cor:cyc} for the study of the complementary pair $(\ebk{\CI}_\loc ,\CC)$ of the Kasparov category $\sigma \Kas ^G$ and the Baum-Connes conjecture (BCC). Oue main interest here is permanence property of the BCC under group extensions, which is studied by Chabert, Echterhoff and Oyono-Oyono in \cite{MR1817505}, \cite{MR1857079} and \cite{MR1836047} with the use of the partial assembly map. Let $1 \to N \to G \xra{\pi} G/N \to 1$ be an extension of groups. It is proved in Corollary 3.4 of \cite{MR1836047} and Theorem 10.5 of \cite{MR2193334} that if $G/N$ and $\pi ^{-1}(F)$ for any compact subgroup $F$ of $G/N$ satisfy the (resp.\ strong) BCC, then so does $G$. Here, the assumption that $\pi ^{-1}(F)$ satisfy the BCC is related to the fact that the assmebly map is defined in terms of the complementary pair $(\ebk{\CI}_\loc ,\CC)$ (this assumption is refined by Schick~\cite{MR2350467} when $G$ is discrete, $H$ is cohomologically complete and has enough torsion-free amenable quotients by group-theoretic arguments). On the other hand, Corollary \ref{cor:cyc} implies that the subcategories $\CC$ and $\CZC$ coincide in $\sigma \Kas ^G$. As a consequence we refine their results as following.

\newtheorem*{intro5}{Theorem \ref{thm:BC}}
\begin{intro5}
Let $1 \to N \to G \to Q \to 1$ be an extension of second countable groups such that all compact subgroups of $Q$ are Lie groups and let $A$ be a $G$-$\Cst$-algebra. Then the following holds:
\begin{enumerate}
\item If $\pi^{-1}(H)$ satisfies the (resp.\ strong) BCC for $A$ for any compact cyclic subgroup $H$ of $Q$, then $G$ satisfies the (resp.\ strong) BCC for $A$ if and only if $Q$ satisfies the (resp.\ strong) BCC for $N \ltimes ^{\mathrm{PR}}_rA$.
\item If $\pi^{-1}(H)$ and $Q$ have the $\gamma$-element for any compact cyclic subgroup $H$ of $Q$, then so does $G$. Moreover, in that case $\gamma _{\pi^{-1}(H)}=1$ and $\gamma _Q=1$ if and only if $\gamma _G=1$.
\end{enumerate}
\end{intro5}

This paper is organized as follows. In Section \ref{section:2}, we briefly summarize terminologies and basic facts on the relative homological algebra of triangulated categories.  In Section \ref{section:3}, we study the relative homological algebra of the injective class in the Kasparov category and prove the Atiyah-Segal completion theorem in $\KK$-theory. Section \ref{section:4} - \ref{section:6} are mutually independent. In Section \ref{section:4}, we study on the restriction map in $\KK$-theory. In Section \ref{section:5} we generalize the Atiyah-Segal completion theorem for groupoid equivariant case. In Section \ref{section:6}, we discuss on permanence property of the Baum-Connes conjecture under extensions of groups. In Appendix \ref{section:App}, we survey definitions and some basic properties of equivariant $\KK$-theory for $\sigma$-$\Cst$-algebras.

\section{Preliminaries in the relative homological algebra}\label{section:2}
In this section we briefly summarize some terminologies and basic facts on the relative homological algebra of triangulated categories. The readers can find more details in \cite{MR2681710} and \cite{MR2563811}. We modify a part of the theory in order to deal with the relative homological algebra of the injective class for countable families of homological ideals. 

A triangulated category is an additive category together with the category automorphism $\Sigma$ so called the suspension and the class of triangles (a sequence $A \xra{f} B \xra{g} C \xra{h} \Sigma A$ such that $g\circ f=h\circ g=\Sigma f \circ h=0$) which satisfies axioms [TR0]-[TR4] (see Chapter 1 of \cite{MR1812507}). We often write an exact triangle $A \to B \to C \to \Sigma A$ as
\[
\xymatrix@R=1.73em@C=1em{A \ar[rr] &&B \ar[ld] \\
&C. \ar[lu]|\circ &}
\]
Here the symbol $\xymatrix@C=1.5em{A \ar[r] |\circ &B}$ represents a morphism from $A$ to $\Sigma B$. 

Let $\Tcat$ be a triangulated category. An \textit{ideal} $\J$ of $\Tcat$ is a family of subgroups $\J (A,B)$ of $\Tcat (A,B)$ such that $\Tcat(A,B)\circ \J(B,C) \circ \Tcat(C,D) \subset \J(A,D)$. A typical example is the kernel of an additive functor $F : \Tcat \to \mathfrak{A}$. We say that an ideal is a \textit{homological ideal} if it is the kernel of a stable homological functor from $\Tcat$ to an abelian category $\mathfrak{A}$ with the suspension automorphism. Here a covariant functor $F$ is homological if $F(A) \to F(B) \to F(C)$ is exact for any exact triangle $A \to B \to C \to \Sigma A$ and stable if $F \circ \Sigma =\Sigma \circ F$. Note that the kernel of an exact functor between triangulated categories is a homological ideal by Proposition 2.22 of \cite{MR2681710}. 

For a homological ideal $\J$ of $\Tcat$, an object $A$ is \textit{$\J$-contractible} if $\id _A$ is in $\J$ and is \textit{$\J$-injective} if $f_*:\Tcat (A,B) \to \Tcat (A,D)$ is zero for any $f \in \J (B,D)$. The triangulated category $\Tcat$ \textit{has enough $\J$-injectives} if for any object $A \in \Obj \Tcat$ there is a $\J$-injective object $I$ and a $\J$-monic morphism $A \to I$ i.e.\ the morphism $\iota$ is the exact triangle $N \xra{\iota} A \to I \to \Sigma N$ is in $\J$. Note that the morphism $\iota$ is \textit{$\J$-coversal}, that is, an arbitrary morphism $f:B \to A$ in $\J$ factors through $\iota$ (see Lemma 3.5 of \cite{MR2563811}). 

More generally, we consider the above homological algebra for a countable family $\J=\{ \J_k \}_{k \in \nat }$ of homological ideals of $\Tcat$. For example, we say an object $A$ is $\J$-contractible if $A$ is $\J_k$-contractible for any $k\in \nat$. 
\begin{defn}\label{def:filt}
A \textit{filtration} associated to $\J$ is a filtration of the morphism sets of $\Tcat$ coming from the composition of ideals $\{ \J_{i_1} \circ \J _{i_2} \circ \cdots \circ \J _{i_r} \} _{r \in \zahl _{>0}}$ where $i_1,i_2,\dots $ is a sequence of positive integers such that each $k \in \nat$ arises infinitely many times. 
\end{defn}
Note that two filtrations associated to $\J$ are equivalent. For simplicity of notation, we use the notation $\J ^r$ for the $r$-th component of a (fixed) filtration associated to $\J$ unless otherwise noted.

The relative homological algebra is related to the complementary pairs (or semi-orthogonal decompositions) of the triangulated categories. For a thick triangulated subcategory $\Cat$ of $\Tcat$ (Definition 1.5.1 and Definition 2.1.6 of \cite{MR1812507}), there is a natural way to obtain a new triangulated category $\Tcat /\Cat$ so called the Verdier localization (see Section 2.1 of \cite{MR1812507}). A pair $(\Ncat, \Icat)$ is a \textit{complementary pair} if $\Tcat (N,I)=0$ for any $N \in \Obj \Ncat$, $I \in \Obj \Icat$ and for any $A \in \Obj \Tcat$ there is an exact triangle $N_A \to A \to I_A \to \Sigma N_A$ such that $N_A \in \Obj \Ncat$ and $I_A \in \Obj \Icat$. Actually, such an exact triangle is unique up to isomorphism for each $A$ and there are functors $N: \Tcat \to \Ncat$ and $I:\Tcat \to \Icat$ that maps $A$ to $N_A$ and $I_A$ respectively. We say that $N$ (resp.\ $I$) the \textit{left} (resp.\ \textit{right}) \textit{approximation functor} with respect to the complementary pair $(\Ncat, \Icat)$. These functors induces the category equivalence $I:\Tcat /\Ncat \to \Icat$ and $N:\Tcat /\Icat \to \Ncat$.

Moreover we assume that a triangulated category $\Tcat$ admits countable direct sums and direct products. A thick triangulated subcategory of $\Tcat$ is \textit{colocalizing} (resp.\ \textit{localizing}) if it is closed under countable direct products (resp.\ direct sums). For a class $\mathcal{C}$ of objects in $\Tcat$, let $\ebk{\mathcal{C}}^{\loc}$ (resp.\ $\ebk{\mathcal{C}}_{\loc}$) denote the smallest colocalizing (resp.\ localizing) thick triangulated subcategory which includes all objects in $\mathcal{C}$. We say that an ideal $\J$ is \textit{compatible with countable direct products} if the canonical isomorphism $\Tcat (A, \prod B_n) \cong \prod \Tcat(A, B_n)$ restricts to $\J (A, \prod B_n) \cong \prod \J(A, B_n)$.

We write $\Ncat _\J$ for the thick subcategory of objects which is $\J _k$-contractible for any $k$. If each $\J_k$ is compatible with countable direct products, $\Ncat _\J$ is colocalizing. We write $\mathfrak{I}_\J$ for the class of $\J_k$-injective objects for some $k$.
\begin{thm}[Theorem 3.21 of \cite{MR2563811}]\label{thm:homideal}
Let $\Tcat$ be a triangulated category with countable direct product and let $\J =\{ \J_i \}$ be a family of homological ideals with enough $\J _i$-injective objects which are compatible with countable direct products. Then, the pair $(\Ncat _{\J}, \ebk{\mathfrak{I}_\J}^{\loc})$ is complementary. 
\end{thm}

We review the explicit construction of the left and right approximation in Theorem 3.21 of \cite{MR2563811}. We start with the following diagram so called the \textit{phantom tower} for $B$:
\[
\xymatrix@R=1.73em@C=1em{
B \ar@{=}[r]&N_0\ar[rd]_{\pi _0} &&N_1 \ar[ll]_{\iota _0^1} \ar[rd]_{\pi_1} &&N_2 \ar[ll]_{\iota _1^2} \ar[rd]_{\pi _2} &&N_3 \ar[ll]_{\iota _2^3} \ar[rd]_{\pi _3} &&N_4 \ar[ll]_{\iota _3^4} \ar[rd]_{\pi _4} &&\cdots \ar[ll]& \\
&& I_0 \ar[ru]|\circ_{\varepsilon _0} \ar[rr]|\circ _{\delta_1} && I_1 \ar[ru]|\circ _{\varepsilon _1} \ar[rr]|\circ _{\delta _2}  && I_2 \ar[ru]|\circ _{\varepsilon_2} \ar[rr]|\circ _{\delta_2} &&  I_3 \ar[ru]|\circ_{\varepsilon _3} \ar[rr]|\circ_{\delta _3}&& \cdots  \\
}
\]
where $\iota _k^{k+1}$ are in $\J _{i_k}$ and $I_k$ are $\J _{i_k}$-injective (here $\{ i_k \}_{k \in \nat}$ is the same as in Definition \ref{def:filt}). There exists such a diagram for any $B$ since $\Tcat$ has enough $\J$-injectives. We write $\iota _k^l$ for the composition $\iota _{l-1} ^l \circ \iota _{l-2}^{l-1} \circ \cdots \circ \iota _{k}^{k+1}$. Since each $\iota _k^{k+1}$ is $\J_{i_k}$-coversal, we obtain $\J^p(A,B)=\Im (\iota _0^p)_*$ for any $A$.

Next we extend this diagram to the \textit{phantom castle}. Due to the axiom [TR1], there is a (unique) object $\tilde{B}_p$ in $\Tcat$ and an exact triangle $N_p \to B \to \tilde{B}_p \to \Sigma N_p$ for each $p$. By the axiom [TR4], we can complete the following diagram by dotted morphisms
\[
\xymatrix@R=1.73em@C=1em{B \ar[rd] && N_{p-1} \ar[ll] \ar[rd] && N_p \ar[ll]\\
&\tilde{B}_{p-1} \ar[ru]|\circ \ar[rr]|\circ && I_p \ar[ru] \ar@{.>}[ld]&\\
&&\tilde{B}_p. \ar@{.>}[lu] &&}
\]
and hence $\tilde{B}_p$ is $\J^p$-injective. Moreover, we obtain a projective system
\[
\xymatrix@R=0.86em@C=1em{
&N_1 \ar[ld]&&N_2 \ar[ld] \ar[ll] &&N_3 \ar[ld] \ar[ll] &&N_4 \ar[ld] \ar[ll] &&N_5 \ar[ld] \ar[ll] && \cdots \ar[ll] \\
B\ar[rd]  && B\ar[rd] \ar[ll]|\hole && B\ar[rd] \ar[ll]|\hole && B\ar[rd] \ar[ll]|\hole && B\ar[rd] \ar[ll]|\hole && \cdots \ar[ll]|\hole &  \\
& \tilde{B}_1 \ar[uu]|\circ&& \tilde{B}_2 \ar[uu]|\circ \ar[ll]&& \tilde{B}_3 \ar[uu]|\circ \ar[ll]&& \tilde{B}_4 \ar[uu]|\circ \ar[ll]&& \tilde{B}_5 \ar[uu]|\circ \ar[ll]&& \cdots \ar[ll]
}
\]
of exact triangles. Now we take the homotopy projective limit $\tilde{B}:=\hoprojlim _p \tilde{B}_p$ and $N:=\hoprojlim N_p$. Here the homotopy projective limit of a projective system $(B,\varphi _m^{m+1})$ is the 
third part of the exact triangle 
$$\Sigma^{-1} \prod B_p \to \hoprojlim B_p \to \prod B_p \xra{\id -S} \prod B_p$$
where $S :=\prod \varphi _m^{m+1}$. Then, the axiom [TR4] implies that the homotopy projective limit $N \to B \to \tilde{B} \to \Sigma N$ of the projective system of exact triangles is also exact. In fact, it can be checked that $\tilde{B}$ is in $\ebk{\Icat _{\J}}^\loc$ and $N$ is in $\Ncat _{\J}$ and hence $N$ and $\tilde{B}$ gives a left and right approximation of $B$.

At the end of this section, we review the ABC spectral sequence, introduced in \cite{MR2563811} and named after Adams, Brinkmann and Christensen. Let $A$ be an object in $\Tcat$, let $\J$ be a countable family of homological ideals with a fixed filtration and let $F: \Tcat \to \Ab$ be a homological functor. Set 
\[
\left\{
\begin{array}{ll}
D=\bigoplus D_{pq},& D_{p,q}:=F_{p+q+1}(N_{p+1}),\\
E=\bigoplus E_{pq},& E_{p,q}:=F_{p+q}(I_p),\\
\end{array}
\right.
\left\{
\begin{array}{ll}
i_{p,q}:=(\iota _p^{p+1})^*&: D_{p,q} \to D_{p+1,q-1},\\
j_{p,q}:=(\varepsilon _p)^*&:D_{p,q} \to E_{p,q+1},\\
k_{p,q}:=(\pi _p)^*&: E_{p,q} \to  D_{p-1,q},
\end{array}
\right.
\]
where $N_{p}=A$ and $I_p=0$ for $p <0$.
Then the triangle
\[
\xymatrix@R=1.73em@C=1em{
D \ar[rr]^i && D \ar[ld]^j \\
&E \ar[lu]^k&
}
\]
forms an exact couple. We call the associated spectral sequence is the \textit{ABC spectral sequence} for $A$ and $F$.

\begin{prp}[Proposition 4.3 of \cite{MR2563811}]\label{prop:ABC}
Let $B$ be an object in $\Tcat$ and let $F$ be a homological functor. Set $D_{pq}^r=D^r_{pq}(B):=i^{r-1}(D_{p+r-1,p-r+1})$ and $E_{pq}^r=E_{pq}^r(B):=k^{-1}(D^r_{pq})/j(\Ker i^r)$. Then the following hold:
\begin{enumerate}
\item 
\[
D^r_{pq}=\left\{
\begin{array}{ll}
\J ^{r-1}F_{p+q+1}(N_p) & \text{if $p \geq 0$}\\
\J ^{p+r-1}F_{p+q+1}(B) & \text{if $-r \leq p \leq 0$}\\
F_{p+q+1}(B) & \text{if $p \leq -r$}
\end{array}
\right.
\]
where $\J^p F(B)$ denotes the subgroup $\{ f_* \xi \mid \xi \in F(A), f \in \J^p (A,B) \}$ of $F(B)$.
\item There is an exact sequence
\ma{0 \to \frac{\J^pF_{p+q+1}(B)}{\J^{p+1}F_{p+q+1}(B)} \to E^\infty_{pq} \to \Bad _{p+1,p+q+1} \xra{i} \Bad _{p,p+q+1}}
where $\Bad _{p,q}(B)=\Bad _{p,q}:=\J ^\infty F_q(N_p)$.
\end{enumerate}
\end{prp}

\begin{lem}\label{lem:comp}
Assume that $i: \Bad _{p+1,p+q+1}(B) \to \Bad_{p,p+q+1}(B)$ is injective. Then, the ABC spectral sequence $E_{pq}^r$ converges to $F(B)$ with the filtration $\J^\ast F(B)$. Moreover, $\alpha_*:F(B) \to F(\tilde B)$ induces an isomorphism of graded quotients with respect to the filtration $\J^*F$.
\end{lem}
\begin{proof}
The convergence of the ABC spectral sequence follows from Proposition \ref{prop:ABC} (2). Since the right approximation is functorial, we obtain the morphism between exact couples which induces the isomorphism $E^2_{pq}(B) \to E^2_{pq}(\tilde{B})$ commuting with the derivation. Hence we obtain the diagram
\[
\xymatrix{
0 \ar[r]& \frac{\J^p F_{p+q+1}(B)}{\J^{p+1}F_{p+q+1} (B)} \ar[r] \ar[d]^{\alpha_*} & E^\infty_{pq}(B) \ar[r] \ar[d]^{\alpha_*}_\cong & \Bad _{p+q+1,p}(B) \ar[r]^i \ar[d]^{\alpha_*} & \Bad _{p,q}(B) \ar[d]^{\alpha_*} \\
0 \ar[r]& \frac{\J ^p F_{p+q+1}(\tilde{B})}{\J^{p+1}F_{p+q+1}(\tilde{B})} \ar[r]& E^\infty_{pq}(\tilde{B}) \ar[r]& \Bad _{p+q+1,p}(\tilde{B}) \ar[r]^i & \Bad _{p,q}(\tilde{B}). \\
}
\]
Consequently, $\alpha _* $ induces the isomorphism of graded quotients of filtered groups $\J^\ast F_p(B)$ and $\J^*F_p(\tilde{B})$ if $i$ is injective.
\end{proof}

\section{The Atiyah-Segal completion theorem}\label{section:3}
In this section we apply the relative homological algebra of the injective class introduced in Section \ref{section:2} for equivariant $\KK$-theory and relate it with the Atiyah-Segal completion theorem. We deal with the Kasparov category $\sigma \Kas ^G$ of $\sigma$-$G$-$\Cst$-algebras, which is closed under countably infinite direct products. The definition and the basic properties of equivariant $\KK$-theory for $\sigma$-$G$-$\Cst$-algebras are summarized in Appendix \ref{section:App}. In most part of this section we assume that $G$ is a compact Lie group. We need not to assume that $G$ is either connected or simply 
connected. 

For a subgroup $H \leq G$, consider the homological ideal $\J_G^H :=\Ker \Res _G^H$ of $\sigma \Kas ^G$. There are only countably many homological ideals of the form $\J_G^H$ since $\J_G^{H_1}=\J_G^{H_2}$ when $H_1 $ and $H_2$ are conjugate and the set of conjugacy classes of subgroups of a compact Lie group $G$ is countable (Corollary 1.7.27 of \cite{MR0177401}),  

\begin{defn}
Let $\F$ be a \textit{family}, that is, a set of closed subgroups of a compact group $G$ that is closed under subconjugacy. We write $\J _G^\F$ for the countable family of homological ideals $\{ \J _G^H \mid H \in \F \}$.
\end{defn}
In particular, we say that the family $\mathcal{T}$ consisting of the trivial subgroup $\{ e \}$ is the trivial family.

By the universal property of the Kasparov category (Theorem \ref{thm:univ}), the induction functor $\Ind _H^G :\sigma \GCsep{H} \to \sigma \GCsep{G}$ given by
$$\Ind _H^G A:=C(G,A)^H=\{ f \in C(G,A) \mid \alpha _h(f(g \cdot h))=f(g) \}$$
with the left regular $G$-action $\lambda_g (f)(g')=f(g^{-1}g')$ induces the functor between Kasparov categories. 

An important property of this functor is the following Frobenius reciprocity.
\begin{prp}[Section 3.2 of \cite{MR2193334}]\label{prp:indres}
Let $G$ be a locally compact group and $H \leq G$ be a cocompact subgroup. Then the induction functor $\Ind _H^G$ is the right adjoint of the restriction functor $\Res _G^H$. That is, for any $\sigma $-$G$-$\Cst$-algebra $A$ and $\sigma $-$H$-$\Cst$-algebra $B$ we have
$$\KK^G(A,\Ind _H^GB) \cong \KK^H(\Res_G^H A,B).$$
\end{prp}
\begin{proof}
The equivariant $\KK$-cycles induced from the $\ast$-homomorphisms
\ma{
\varepsilon _A :\Res _G^H \Ind _H^G A \cong C(G,A)^H \to A;& f \mapsto f(e)\\
\eta _B: B \to \Ind _H^G \Res _G^H B \cong C(G/H)\otimes B; & a \mapsto a \otimes 1_{G/H}
}
form a counit and a unit of an adjunction between $\Ind _H^G$ and $\Res _G^H$. Actually it directly follows from the definition that the compositions 
\ma{
&\Res _G^H A \xra{\Res _G^H \eta _A} \Res _G^H \Ind _H^G \Res _G^H A \xra{\varepsilon _{\Res _G^HA}}\Res _G^H A \\
&\Ind _H^GB \xra{\eta _{\Ind _H^GB}}\Ind _H^G \Res _G^H \Ind _H^G B \xra{\Ind _H^G \varepsilon _B}\Ind _H^GB}
are identities in $\sigma \Kas ^G$. 
\end{proof}

\begin{defn}\label{def:FindFres}
Let $G$ be a compact group and let $\F$ be a family of $G$.
\begin{enumerate}
\item A separable $\sigma$-$G$-$\Cst$-algebra $A$ is \textit{$\F$-induced} if $A$ is isomorphic to the inductions $\Ind _H^G A_0$ where $A_0$ is a separable $\sigma$-$H$-$\Cst$-algebra and $H \in \F$. 
We write $\FI$ for the class of $\F $-induced objects. 
\item A separable $\sigma$-$G$-$\Cst$-algebra $A$ is \textit{$\F$-contractible} if $\Res _{G}^H A$ is $\KK^H$-contractible for any $H \in \F$. We write $\FC$ for the class of $\mathcal{F}$-contractible objects.
\end{enumerate}
In particular, when $\F =\mathcal{T}$ we say that $A$ is trivially induced and trivially contractible respectively. 
\end{defn}

\begin{thm}\label{thm:decomp}
Let $G$ be a compact group and let $\F$ be a family $G$. The pair $(\FC,\ebk{\FI}^{\loc})$ is complementary in $\sigma \Kas^G$.
\end{thm}
\begin{proof}
This is proved in the same way as Proposition 3.37 of \cite{MR2681710}. By definition, we have $\FC = \Ncat _{\J_G^\F}$ and $\FI \subset \Icat _{\J_G^\F}$. Therefore, by Theorem \ref{thm:homideal}, it suffices to show that $\sigma \Kas ^G$ has enough $\J_G^{\F}$-injectives and all $\J_G^ \F$-injective objects are in $\ebk{\FI}^{\loc}$. The first assertion follows from the existence of the right adjoint functor of $\Res _G^H$. Actually, for any $H \in \F$, the morphism $A \to I_1:=\Ind _H^G\Res _G^HA$ is $\J _G^H$-monic and $I_1$ is $\J_G^H$-injective. Moreover, the morphism $A$ is a direct summand of $I_1$ when $A$ is $\J_G^H$-injective. This implies the second assertion.
\end{proof}

In particular, applying Theorem \ref{thm:decomp} for the case of $\F =\mathcal{T}$, we immediately get the following simple but non-trivial application.
\begin{cor}\label{cor:hom}
Let $A$ be a separable $\sigma$-$\Cst$-algebra and let $\{ \alpha _t \} _{t \in [0,1]}$ be a homotopy of $G$-actions on $A$. We write $A_t$ for the $\sigma $-$G$-$\Cst$-algebra $(A,\alpha _t)$. Then, $A_0$ and $A_1$ are equivalent in $\sigma \Kas ^G/\mathcal{TC}$. In particular, if $A_0$ and $A_1$ are in $\ebk{\mathcal{TI}}^\loc$, then they are $\KK ^G$-equivalent.
\end{cor}
Corollary \ref{cor:hom} is applied for the study of $\Cst$-dymniaical systems in the upcoming paper \cite{AranoKubota}. Actually, it follows from Thomsen's description of $\KK$-groups using completely positive asymptotic morphisms \cite{MR1680467} that a unital $G$-$\Cst$-algebra with continuous Rokhlin property (or more generally finite continuous Rokhlin dimension with commuting tower) is contained in the subcategory $\ebk{\mathcal{TI}}^\loc$. 

\begin{proof}
Consider the $\sigma$-$G$-$\Cst$-algebra $\tilde{A}:=(A \otimes C[0,1],\tilde{\alpha})$ where $\tilde{\alpha}(a)(t)=\alpha _t(a(t))$. Since the evaluation maps $\ev _t: \tilde{A} \to A_t$ are non-equivariantly homotopy equivalent, they induce equivalences in $\sigma \Kas ^G / \mathcal{TC}$. Consequently, $\ev _1 \circ (\ev _0)^{-1}:A_0 \to A_1$ is an equivalence in $\sigma \Kas ^G / \mathcal{TC}$. The second assertion is obvious.
\end{proof}

Next we study a canonical model of phantom towers and phantom castles. Actually, we observe that the cellular approximation tower obtained in the proof of Theorem \ref{thm:decomp} is nothing but the Milnor construction of the universal $\F$-free $G$-space (see \cite{MR2195456}). Hereafter, for a compact $G$-space $X$, we write $\cone _X$ for the mapping cone $\{ f \in C_0([0,\infty), C(X)) \mid f(0)=\comp \cdot 1_X \}$ of the $\ast$-homomorphism $\comp \to C(X)$ induced from the collapsing map $X \to \mathrm{pt}$. 

\begin{defn}\label{def:Milnor}
Let $\{H_p\}_{p \in \zahl _{>0}}$ be a countable family of subgroups in $\F$ such that any $L \in \F$ are contained infinitely many $H_p$'s. We call the phantom tower and the phantom castle determined inductively by
$$I_p:=\Ind_H^G \Res_G^H N_{p-1} \cong N_{p-1} \otimes C(G/H_p)$$
is the \textit{Milnor phantom tower} and the \textit{Milnor phantom castle} (associated to $\{ H_p\}$) respectively. 
\end{defn}

By definition, $I_k$ and $N_k$ in the Milnor phantom tower are explicitly of the form
\ma{N_k &\cong A \otimes \cone_{G/H_1} \otimes \cdots \otimes \cone_{G/H_k}\\
I_k &\cong A \otimes \cone_{G/H_1} \otimes \cdots \otimes \cone_{G/H_{k-1}} \otimes C(G/H_k)}
and $\iota _k^{k+1}$ is induced from the restriction (evaluation) $\ast$-homomorphism $\ev _0 : \cone _{G/H_k} \to \comp$ given by $f \mapsto f(0)$.

\begin{lem}\label{lem:Milnor}
The $n$-th step of the cellular approximation $\tilde{\comp}_n$ of $\comp$ is isomorphic to the $n$-th step of the Milnor construction $C(\bigast _{k=1}^n G/H_k)$. 
\end{lem}
\begin{proof}
The join $\bigast _{k=1}^n G/H_k$ is defined as the quotient $\Delta^n \times (\prod G/H_k) /\sim $ where $\Delta ^n:=\{ (t_1,\dots ,t_n) \in [0,1]^n \mid \sum t_i=1 \}$ and $(t_1,\dots,t_n,x_1,\dots,x_n) \sim (t_1,\dots, t_n,y_1,\dots,y_n)$ if $x_k=y_k$ for any $k$ such that $t_k \neq 0$. 
Let $f$ be an element in $\cone _{G/H_1} \otimes \cdots \otimes \cone _{G/H_n}$ given by 
$$f(((x_1,t_1),\dots,(x_n,t_n)))=t_1+\dots +t_n.$$
By definition $f^{-1}(t)$ is $G$-homeomorphic to the join $\bigast G/H_k$ and moreover $f^{-1}((0,\infty))\cong (0,\infty) \times (\bigast G/H_k)$. On the other hand, $f^{-1}(0)=\ast$. Consequently $\cone _{G/H_1} \otimes \cdots \otimes \cone _{G/H_n}$ is $G$-equivariantly isomorphic to the mapping cone $\cone _{\bigast  G/H_k}$.
\end{proof}

More generally, let $X$ be a $\F$-free (i.e.\ every stabilizer subgroups are in $\F$) finite $G$-CW-complex containing a point $x$ whose stabilizer subgroup is $H$. By Proposition 2.2 of \cite{MR2563811}, there is $n>0$ such that $C(X)$ is $(\J _G^\F)^n$-injective. Moreover, the morphism $\ev_0 : \cone _X \to \comp$ is in $\J_G^H$ since the path of $H$-equivariant $\ast$-homomorphisms $\ev _{(t,x)} : \cone _X \to \comp $ connects $\ev _0$ and zero. Let $\{ X_i \}$ be a family of $\F$-free compact $G$-CW-complexes such that for any $H \in \F$ there are infinitely many $X_i$'s such that $X_i^H\neq \emptyset$. Then, in the same way as Theorem \ref{thm:homideal}, the exact triangle 
$$SC(\bigast_{i=1}^\infty X_i) \to C_0(\prod_{i=1}^\infty \cone _{X_i}) \to \comp \to C(\bigast _{i=1}^\infty X_i)$$
gives the approximations of $\comp$ with respect to the complementary pair $(\FC, \ebk{\FI}^\loc)$.  

Now we compare the filtration $(\J_G^\F)^*(A,B)$ with another one;
$$(I_G^\F)^n \KK^G(A,B):=\{ \sum \gamma _i^1 \cdots \gamma _i^n \xi_i \mid \gamma _k^i \in I_G^{H_k}, \xi_i \in \KK^G(A,B) \}$$ 
where $I_G^{H}$ are the augmentation ideals $\Ker \Res _G^H$ of $R(G)$ and $\{ H_i \}$ is the same as Definition \ref{def:Milnor}. Obviously its equivalence class is independent of the choice of such $\{ H_i \}$. 
\begin{exmp}\label{exmp:torus}
We consider the case that $G=\mathbb{T}^1$ and $\F =\mathcal{T}$. The first triangle in the Milnor phantom tower is
\[
\xymatrix@R=1.73em@C=1em{
\comp \ar[rd] && C_0(\real ^2) \ar[ll]_{\iota _0^1} \\
&C(\mathbb{T}^1) \ar[ru]|\circ &
}
\]
where $\mathbb{T}^1=U(1)$ acts on $\real ^2 =\comp$ canonically. By the Bott periodicity, $\KK^G(N_1,\comp)$ is freely generated by the Bott generator $\beta \in \KK^G(N_1,\comp)$ and $\J_G (N_1,\comp)=I_G \cdot \beta$. Consequently, $\iota _0^1 $ is in $I_G\KK^G(A,B)$. More explicitly, $\iota _0^1=\lambda \cdot \beta$ where $\lambda :=[\Lambda ^0 \comp]-[\Lambda ^1 \comp]$. Since $\iota _0^1$ is $\J_G$-coversal, $\J _G(A,B) =I_G\KK^G(A,B)$ holds for any $A$ and $B$. 
\end{exmp}

\begin{exmp}\label{exmp:flag}
Let $G$ be a Lie group with Hodgkin condition (i.e.\ $G$ is connected and $\pi _1(G)$ is torsion free) and let $T$ be a maximal torus of $G$. By the Borel-Weil-Bott theorem, the equivariant index of the Dolbeault operator $\overline{\partial}+\overline{\partial}^* $ on the flag manifold $G/T$ is $1 \in R(G)$. Therefore, the corresponding $\K$-homology cycle $[\overline{\partial}+\overline{\partial}^*]$ determines a left inverse of $\pi^* : \comp \to C(G/T)$. This implies that $\iota _0^1=0$. More generally, for any compact Lie group $G$, there is a subgroup $T$ of $G$ which is isomorphic to a finite extension of a torus such that $\comp$ is a direct summand of $C(G/T)$ and hence $\J _G^T=I_G^T\KK^G=0$ (Proposition 4.1 of \cite{MR0248277}). 
\end{exmp}

\begin{thm}\label{thm:IG1}
Let $H \leq G$ be compact Lie groups satisfying the Hodgkin condition and $\rank G - \rank H\leq 1$. For a group homomorphism $\varphi: L \to G$, let $\F$ be the smallest family containing $\{ \varphi ^{-1}(gHg^{-1}) \mid g \in G \}$. Then, for any $r \in \zahl _{>0}$ there is $k \in \zahl _{>0}$ such that $\iota ^k_0 \in (I_L^\F)^r \KK^L (N_k,\comp)$.
\end{thm}

\begin{proof}
Let $(N_l, I_l)$ and $(N_k', I_k')$ be a Milnor phantom tower of $\comp$ in $\sigma \Kas ^L$ and $\sigma \Kas ^G$ respectively. Since $L$ acts $\F$-freely on $\bigast _{i=1}^k G/H_i$ by $\varphi$, for any $k>0$ there is $l>0$ such that $\varphi ^*I_k'$ is $(\J _L^\F)^l$-injective. Thus, the composition $N_l \to \comp \to \varphi^*I_k'$ is zero and hence $\iota _0^l :N_l \to \comp$ factors through $\varphi ^*\iota _0^k : \varphi ^*N_k' \to \comp$. Therefore, it suffices to show the assertion when $\varphi =\id$.

When $\rank G =\rank H$, it immediately follows from Example \ref{exmp:flag}. To see the case that $\rank G-\rank H =1$, choose an inclusion of maximal tori $T_H \subset T_G$. Consider the exact triangle 
$S C(T_G/T_H) \to \cone _{T_G/T_H} \to \comp \to C(T_G/T_H)$. In the same way as Example \ref{exmp:torus}, we obtain that $\Res _G^{T_G}\iota _0^1$ is in $I_{T_G}^{T_H}\KK^{T_G}(N_1,\comp)$. Since $(I_{T_G}^{T_H})^n \subset I_G^{T_H}R(T_G)$ for sufficiently large $n>0$ (Lemma 3.4 of  \cite{MR935523}), for any $l>0$ there is $k>0$ such that $\iota _0^k=\iota _0^1 \otimes \cdots \otimes \iota _0^1$ is in $(I_{G}^{H})^l \KK^{T_G}(N_k,\comp)$ (note that $I_G^{T_H}=I_G^H$). We obtain the consequence because $\KK^G(A,B)$ is a direct summand of $\KK^{T_G}(A,B)$.
\end{proof}

As a corollary, we obtain a generalization of Corollary 1.3 of \cite{MR935523}. For a family $\F$ of $G$, we write $\F _{\cyc}$ for the family generated by (topologically) cyclic subgroups in $\F$. In particular, let $\Z$ denote the family generated by all cyclic subgroups. Here, we say that $T \leq G$ is a cyclic subgroup of $G$ if there is an element $g \in T$ such that $\overline{\{ g^n \}}=T$. Note that $T$ is cyclic if and only if $T \cong \mathbb{T}^m \times \zahl /l\zahl$.

\begin{cor}\label{cor:cyc}
For general compact Lie group $G$, the following hold: 
\begin{enumerate}
\item There is $n>0$ such that $(\J ^{\Z}_G)^n=0$. In particular, the subcategory $\Z \mathcal{C}$ is zero in $\sigma \Kas ^G$.
\item For any family $\F$ of $G$, the filtrations $(\J_G^\F)^*$ and $(\J_G^{\F_{\cyc}})^*$ are equivalent. Moreover, $\FC= \F _\cyc \mathcal{C}$ in $\sigma \Kas ^G$.
\end{enumerate}
\end{cor}
Note that the second assertion means that for any $n>0$ we obtain $k>0$ (which does not depend on $A$ and $B$) such that $(\J_G^\F)^k(A,B) \subset (\J_G^{\F_{\cyc}})^n(A,B)$.
\begin{proof}
Let $\pi : G \to U(n)$ be a faithful representation of $G$. Apply Theorem \ref{thm:IG1} for $T_{U(n)} \leq U(n)$ and $\pi$. In this case $\F$ is equal to the family of all abelian subgroups $\mathcal{AB}$ of $G$. Consequently we obtain $k\in \zahl_{>0}$ such that $(\J_G^{\mathcal{AB}})^k=0$. Therefore, it suffices to show that for any abelian compact Lie group $G$ there is a large $n>0$ such that $(\J_G^\Z )^n(A,B)=0$. 

We prove it by induction with respect to the order of $G/G^0$. When $G/G^0$ is cyclic, then the assertion holds because $G$ is also cyclic. Now we assume that $G/G^0$ is not cyclic (and hence any element in $G/G^0$ is contained in a proper subgroup). Let $\mathcal{P}$ be the family of $G$ generated by pull-backs of proper subgroups of $G/G^0$. By the induction hypothesis, it suffices to show that there is a large $n>0$ such that $(\J _G^{\mathcal{P}})^n=0$. Because $G$ is covered by finitely many subgroups in $\mathcal{P}$, we obtain a large $m>0$ such that $(I_G^{\mathcal{P}})^m=0$. Applying Theorem \ref{thm:IG1} for compositions of the quotient $\pi :G \to G/G^0$ and group homomorphisms $G/G^0 \to \mathbb{T}^1$, we obtain $n>0$ such that $(\J _G^{\mathcal{P}})^n \subset (I_G^{\mathcal{P}})^m\KK^G=0$. 

The assertion (2) immediately follows from (1). 
\end{proof}

\begin{remk}
Unfortunately, in contrast to Theorem \ref{thm:IG1}, $\iota _0^k \in I_G^\F \KK^G(N_k,\comp)$ does not hold for general compact Lie groups and families. For example, consider the case that $G=\mathbb{T}^2$ and $\F=\mathcal{T}$. Computing the six-term exact sequence of the equivariant $\K$-homology groups associated to the exact triangle
$$SC(S^{2n-1} \times S^{2n-1}) \to \cone _{S^{2n-1} \times S^{2n-1}} \to \comp \to C(S^{2n-1} \times S^{2n-1}),$$
we obtain $\KK ^G(\cone _{S^{2n-1} \times S^{2n-1}},\comp) \cong R(G) \cdot \iota _0^k$ (note that $\KK^G_1(C(S^{2n-1} \times S^{2n-1}),\comp) \cong \K_1(\comp P^{n} \times \comp P^n)=0$ by Poincar\'e duality). By Theorem \ref{thm:limit} (3), we obtain $\KK^G(N,\comp) \cong R(G) \cdot \iota _0^\infty $ and hence $\iota _0^\infty$ is not in $I_G \KK^G(N,\comp)$. 
\end{remk}

Instead of Theorem \ref{thm:IG1}, the following theorem holds for general compact Lie groups and families.

\begin{thm}\label{thm:IG2}
Let $G$ be a compact Lie group and let $A$, $B$ be $\sigma$-$\Cst$-algebras such that $\KK^G_*(A,B)$ is finitely generated for $\ast=0,1$. Then the filtrations $(\J_G^\F)^*(A,B)$ and $(I_G^\F)^* \KK^G(A,B)$ are equivalent.
\end{thm}
Note that this is a direct consequence of Lemma \ref{lem:Milnor} and Corollary 2.5 of \cite{MR2887199} when $\KK^H_*(A,B)$ are finitely generated for any $H \leq G$ and $\ast=0,1$.

To show Theorem \ref{thm:IG2}, we prepare some lemmas.
\begin{lem}\label{lem:fingen}
Let $G$ be a compact Lie group, let $X$ be a compact $G$-space and let $A$, $B$ be $\sigma$-$G \ltimes X$-$\Cst$-algebras. We asssume that $\KK^{G \ltimes X}_*(A,B)$ are finitely generated for $\ast=0,1$. Then, the following holds:
\begin{enumerate}
\item Assume that $G$ satisfies Hodgkin condition and let $T$ be a maximal torus of $G$. Then $\KK^{T \ltimes X}_*(A,B)$ are finitely generated for $\ast=0,1$.
\item When $G=\mathbb{T}^n$, $\KK^{H \ltimes X}_*(A,B)$ are finitely generated for any $H \leq \mathbb{T}^n$.
\item For any cyclic subgroup $H$ of $G$, there is a $G$-space $Y$ such that $C(Y)$ is $(\J_G^H)^k$-injective for some $k>0$ and $\KK^{G \ltimes X}_*(A,B \otimes C(Y))$ are finitely generated for $\ast =0,1$.
\end{enumerate}
\end{lem}
\begin{proof}
First, (1) follows from the fact that $C(G/T)$ is $\KK^G$-equivalent to $\comp ^{|W_G|}$ (which is essentially proved in p.31 of \cite{MR849938}). To see (2), first we consider the case that $\mathbb{T}^n/H$ is isomorphic to $\mathbb{T}$. Then, the assertion follows from the six-term exact sequence of the functor $\KK^{\mathbb{T}^n \ltimes X}(A,B \otimes \blank)$ associated to the exact triangle
$$SC(\mathbb{T}^1) \to C_0(\real ^2) \to \comp \to C(\mathbb{T}^1).$$
In general $\mathbb{T}^n/H$ is isomorphic to $\mathbb{T}^m$. By iterating this argument $m$ times, we immediately obtain the conclusion.

Finally we show (3). Since the space of conjugacy classes of $G$ is homeomorphic to the quotient of a finite copies of the maximal torus $T$ of $G_0$ by a finite group, there is a finite family of class functions separating conjugacy classes of $G$. A moment thought will give you a finite faithful family of representations $\{ \pi _i:G \to U(n_i) \}$ such that $\{ \chi (\pi _i) \}$ separates the conjugacy classes of $G$. Then, two elements $g _1$ , $g_ 2$ in $G$ are conjugate in $G$ if and only if so are in $U:=\prod U(n_i)$ (here $G$ is regarded as a subgroup of $U$ by $\prod \pi _i$). Set $\F := \{ L \leq G \cap gHg^{-1} \mid g \in U\}$. Then $G$ acts on $U /H$ $\F$-freely and every subgroup in $\F _\cyc$ is contained in a conjugate of $H$. By Corollary \ref{cor:cyc} (2), $C(U/H)$ is $(\J_G^H)^k$-injective for some $k>0$. Moreover, $\KK^G_*(A,B \otimes C(U/H))$ are finitely generated $R(G)$-modules. To see this, choose a maximal torus $T$ of $U$ containing $H$. Then $U/H$ is a principal $T/H$-bundle over $U/T$ and we can apply the same argument as (2).
\end{proof}

\begin{lem}\label{lem:local}
Let $X$ be a compact $G$-space and let $X_1,\dots ,X_n$ be closed $G$-subsets of $X$ such that $X_1 \cup \cdots \cup X_n=X$. Then, in the category $\sigma \Kas ^{G \ltimes X}$, the filtration associated to the family of ideals $\J _{X_1,\dots ,X_n}:=\{ \Ker \Res _{G \ltimes X}^{G \ltimes X_i} \}$ is trivial (i.e.\ there is $k>0$ such that $(\J _{X_1,\dots ,X_n})^k=0$).
\end{lem}
\begin{proof}
It suffices to show the following: Let $X$ be a compact $G$-space and $X_1$, $X_2$ be a closed $G$-subspaces such that $X=X_1 \cup X_2$. For separable $\sigma$-$G \ltimes X$-$\Cst$-algebras $A$, $B$, $D$ and $\xi _1 \in \KK^{G \ltimes X}(A,B)$, $\xi _2 \in \KK^{G \ltimes X}(B,D)$ such that $\Res _{G \ltimes X} ^{G  \ltimes X_1}\xi _1 =0$ and $\Res _{G \ltimes X} ^{G  \ltimes X_2} \xi _2=0$ holds, we have $\xi _2 \circ \xi _1 =0$. 

To see this, we use the Cuntz picture. Let $\Kop _G :=\Kop (L^2(G)^\infty)$ and let $q_{s,X}A$ be the kernel of the canonical $\ast$-homomorphism 
$$((A \otimes \Kop _G) \ast _X (A \otimes \Kop _G)) \otimes \Kop _G \to (A \otimes \Kop _G) \otimes \Kop _G$$
for a $G \ltimes X$-$\Cst$-algebra $A$. Then, $\KK^{G \ltimes X}(A,B)$ is isomorphic to the set of homotopy classes of $G \ltimes X$-equivariant $\ast$-homomorphisms from $q_{s,X}A$ to $q_{s,X}B$ and the Kasparov product is given by the composition. 

Let $X'$ be the $G$-space $X_1 \times \{ 0 \}  \cup (X_1 \cap X_2) \times [0,1] \cup X_2 \times \{ 1 \} \subset X \times [0,1]$ and let $p:X' \to X$ be the projection. Note that $p$ is a homotopy equivalence. Let $\varphi _1: q_{s,X}A \to q_{s,X}B$ be a $G \ltimes X$-equivariant $\ast$-homomorphism such that $[\varphi _1]=\xi _1$. By using a homotopy trivializing $\varphi _1|_{X_1}$, we obtain a $G \ltimes X'$-equivariant $\ast$-homomorphism $\varphi _1':q_{s,X'}p^*A\to q_{s,X'}p^*B$ such that $[\varphi _1']=\xi_1$ under the isomorphism $\KK^{G \ltimes X}(A,B) \cong \KK^{G \ltimes X'}(p^*A,p^*B)$ and $\varphi _1'=0$ on $X' \cap X \times [0,1/2]$. Similarly, we get $\varphi _2': p^*q_sB \to p^*q_sD$ such that $[\varphi _2'] =\xi _2$ and $\varphi _2'=0$ on $X' \cap X \times [1/2,1]$. Then, $\xi _2 \circ \xi _1=[\varphi _2' \circ \varphi _1']=0$.
\end{proof}

\begin{proof}[Proof of Theorem \ref{thm:IG2}]
By Corollary \ref{cor:cyc}, we may replace $\F$ with $\F _{\mathrm{cyc}}$. When $G=\mathbb{T}^n$, the conclusion follows from Lemma \ref{lem:fingen} (2) and Corollary 2.5 of \cite{MR2887199}.

For general $G$, let $U$ be the Lie group as in the proof of Lemma \ref{lem:fingen} (3) and let $T$ be a maximal torus of $U$. Consider the inclusion
\ma{\KK^G(A,B) &\cong \KK^{U \ltimes U/G}(\Ind _G^UA, \Ind _G^UB) \\
& \subset KK^{T \ltimes U/G}(\Ind _G^UA , \Ind _G^UB ).}
Set $\tilde{\F}$ and $\F'$ the family of $G$ and $T$ respectively given by 
$$\tilde{\F}:=\{L \leq G \cap gHg^{-1} \mid H \in \F, g \in U \}, \ \F':=\{ L \leq T \cap gHg^{-1} \mid H \in \F , g \in U \}.$$
Note that Corollary \ref{cor:cyc} implies that the filtration $(\J^{\tilde{F}}_G)^*$ is equivalent to $(\J _G^\F)^*$ since $\F _{\cyc}=\tilde{\F}_\cyc$.

Consider the family of homological ideals
$$\J_{T \ltimes U/G}^{\F'}:=\{ \Ker \Res _{T \ltimes U/G}^{H \ltimes U/G} \mid H \in \F' \}.$$
We claim that the restriction of the filtration $(\J _{T \ltimes U/G}^{\F'})^*(\Ind _G^UA , \Ind _G^UB )$ on $\KK^G(A,B)$ is equivalent to $(\J_G^\F)^*(A,B)$.

Pick $L \in \F'$. The slice theorem (Theorem 2.4 of \cite{MR2292634}) implies that there is a family of closed $L$-subspaces $X_1,\dots ,X_n$ of $U/G$ and $x _i \in X_i$ such that $\bigcup X_i=U/G$ and the inclusions $Lx_i \to X_i$ are $L$-equivariant homotopy equivalences. Now we have canonical isomorphisms 
\ma{\KK^{L \ltimes X_i}(\Ind _G^UA|_{X_i},\Ind _G^UB|_{X_i}) \xra{\Res _{X_i}^{Lx_i}}& \KK^{L \ltimes L x_i}(\Ind _G^UA|_{L x_i},\Ind _G^UB|_{L x_i})\\
\to& \KK^{gLg^{-1} \cap G}(A,B)}
such that $\Res _G^{gLg^{-1}\cap G}=\Res _{U \ltimes U/G}^{L \ltimes X_i}$ under these identifications (here $g \in U$ such that $gL=x_i \in U/L$). Now, we have $gLg^{-1}\cap G \in \tilde{\F}$. Therefore, by Lemma \ref{lem:local}, we obtain $(\J^{\tilde \F}_G)^k \subset \J^{\F'}_{T \ltimes U/G}$ for some $k>0$. Conversely since $\F = \F_{\cyc}$, for any $L \in \tilde {\F}$, we can take $g \in U$ such that $g L g^{-1} \in \F'$. Hence $\KK^G(A,B) \cap \J^{\F'}_{T \rtimes U/G}(A,B) \subset \J^{\tilde \F}_G(A,B)$.

Similarly, the filtration $(I_G^\F)^*\KK^G(A,B)$ is equivalent to the restriction of $(I_T^{\F '})^*\KK^{T \ltimes U/G}(\Ind _G^UA,\Ind _G^UB)$. Actually, by Lemma 3.4 of \cite{MR935523}, the $I_G^\F$-adic and $I_U^{\F''}$-adic topologies on $\KK^G(A,B)$ (here $\F''$ is the smallest family of $U$ containing $\F'$) coincide and so do the $I_U^{\F'}$-adic and $I_T^{\F'}$-adic topologies on $\KK^{T \ltimes U/G}(\Ind _G^UA,\Ind _G^U B)$. 

Finally, the assertion is reduced to the case of $G=\mathbb{T}^n$.
\end{proof}

Theorem \ref{thm:IG2} can be regarded as a categorical counterpart of the Atiyah-Segal completion theorem. Since Theorem \ref{thm:IG2} holds without assuming that $\KK^H_*(A,B)$ are finitely generated for every $H \leq G$, we also obtain a refinement of the Atiyah-Segal theorem (Corollary 2.5 of \cite{MR2887199}).

\begin{lem}\label{lem:proisom}
Let $A,B$ be separable $\sigma$-$G$-$\Cst$-algebras such that $\KK^G_*(A,B)$ are finitely generated for $\ast=0,1$. Then there is a pro-isomorphism
$$\{ \KK^G(A,B)/(\J_G^\F)^p(A,B) \}_{p \in \zahl _{>0}} \to \{ \KK^G(A,\tilde{B}_p) \}_{p \in \zahl _{>0}}.$$
\end{lem}
\begin{proof}
By Lemma \ref{lem:fingen} (3), there are compact $G$-spaces $\{ X_k \}_{k \in \zahl _{>0}}$ such that $\KK^G_*(A,B \otimes C(X_k))$ are finitely generated for $\ast=0,1$, each $C(X_i)$ is $(\J_G^\F)^r$-injective for some $r>0$ and for any $H \in \F$ there are infinitely many $X_k$'s such that $X_k^H\neq \emptyset$. Set
$$N_p':= B \otimes \bigotimes _{i=1}^p \cone _{X_i},\ I_p':=N_{p-1}' \otimes C(X_p),\ \tilde{B}_p':=B \otimes C(\bigast _{i=1}^p X_i)$$
and $N':=\hoprojlim N_p'$, $\tilde{B}':=\hoprojlim \tilde{B}'_p$. By the same argument as Theorem \ref{thm:homideal}, we obtain that
$$S\tilde{B} \to N \to B \to \tilde{B}$$
is the approximation of $B$ with respect to $(\FC,\ebk{\FI}^\loc)$. Moreover, by the six-term exact sequence, we obtain that $\KK^G_*(A,\tilde{B}_p')$ are finitely generated $R(G)$-modules.

Consider the long exact sequence of projective systems
$$\{ \KK^G_*(A,S\tilde{B}'_p) \}_p \xra{\partial _p} \{ \KK^G_*(A,N_p') \}_p \xra{(\iota_0^p)_*} \{ \KK^G_*(A,B) \}  \xra{(\alpha _0^p)_*} \{ \KK^G_*(A,\tilde{B}'_p) \}_p.$$
Then, $\{ \Im (\iota _0^p)_* \}_p = \{ \Ker (\alpha _0^p)_* \} _p$ is pro-isomorphic to $(\J_G^\F)^\ast (A,B) $. Actually, for any $p>0$ there is $r>0$ such that $(\J_G^\F)^r (A,B) \subset \Ker (\alpha_0^p)_* =\Im (\iota _0^p)_* \subset (\J _G^\F)^p (A,B)$ since $\tilde{B}_p'$ is $(\J_G^\F)^r$-injective for some $r>0$.

Therefore, it suffices to show that the boundary map $\{ \partial_p \}$ is pro-zero. Apply Theorem \ref{thm:IG2} and the Artin-Rees lemma for finitely generated $R(G)$-modules $M:=\KK^G(A,N_p')$ and $N:=\Im  \partial _p$. Since $\tilde{B}_p'$ is $(\J_G^\F)^r$-injective for some $r >0$, there is $k>0$ and $l>0$ such that 
$$\Im (\iota _p^{p+l})_* \cap N=(\J_G^\F)^l(A,N_p') \cap N \subset (I_G^\F)^{k}M \cap N \subset (I_G^\F)^r N=0.$$
Consequently, for any $p>0$ there is $l>0$ such that $\Im \iota _p^{p+l} \circ \partial _{p+l}=0$.
\end{proof}

\begin{remk}
It is also essential for Lemma \ref{lem:proisom} to assume that $\KK^G_*(A,B)$ are finitely generated. Actually, by Theorem \ref{thm:IG1}, the pro-isomorphism in Lemma \ref{lem:proisom} implies the completion theorem when $G=\mathbb{T}^1$ and $\F=\mathcal{T}$. On the other hand, since the completion functor is not exact in general, there is a $\sigma$-$\Cst$-algebra $A$ such that the completion theorem fails for $\K_*^G(A)$. For example, let $A$ be the mapping cone of $\oplus \lambda ^n : \oplus \comp \to \oplus \comp$. Then, the completion functor for the exact sequence $0 \to \oplus R(G) \to \oplus R(G) \to K_0^G(A) \to 0$ is not exact in  the middle (cf. Example 8 of \cite[Chapter 86]{StackProject} ). 
\end{remk}

\begin{lem}\label{lem:ABC}
Let $A,B$ be separable $\sigma$-$G$-$\Cst$-algebras such that $\KK^G_*(A,B)$ are finitely generated for $\ast=0,1$. Then, the ABC spectral sequence for $\KK^G(A,\blank)$ and $B$ converges toward $\KK^G(A,B)$ with the filtration $(\J_G^\F)^*(A,B)$.
\end{lem}
\begin{proof}
According to  Lemma \ref{lem:comp}, it suffices to show that $i :\Bad _{p+1,p+q+1} \to \Bad _{p,p+q+1}$ is injective. As is proved in Lemma \ref{lem:proisom}, the boundary map $\partial _p$ is pro-zero homomorphism and hence the projective system $\{ \Ker \iota _0^p \} = \{ \Im \partial_p \}$ is pro-zero. Therefore, for any $p>0$ there is a large $q>0$ such that 
$$\Ker \iota _0^1 \cap (\J _G^\F)^\infty (A,N_p) \subset \Ker \iota_0^p \cap (\J _G^\F)^q (A,N_p) = \Ker \iota _0^p \cap \Im \iota _p^{p+q} =0.$$
\end{proof}

\begin{thm}\label{thm:AS}
Let $A$ and $B$ be separable $\sigma$-$G$-$\Cst$-algebras such that $\KK^G_*(A,B)$ are finitely generated $R(G)$-modules ($\ast=0,1$). Then, the morphisms
\begin{itemize}
\item $\KK^G(A,B) \to \KK^G(A,\tilde{B})$,
\item $\KK^G(A,B) \to \RKK^G(E_\mathcal{F}G;A,B)$,
\item $\KK^G(A,B) \to  \sigma \Kas ^G/\FC (A,B)$
\end{itemize}
induce the isomorphism of graded quotients with respect to the filtration $(\J_G^\F)^*(A,B)$. In particular, we obtain isomorphisms
$$\KK ^G(A,B) _{I_G^\F}^{\myhat} \cong \KK^G(A,\tilde{B}) \cong \RKK^G(E_\F G;A,B) \cong \sigma \Kas ^G/\FC (A,B).$$
\end{thm}
\begin{proof}
This is a direct consequence of Lemma \ref{lem:proisom} and Lemma \ref{lem:ABC}. Note that Lemma \ref{lem:proisom} implies that the projective system $\{ \KK ^G(A,\tilde{B}_p) \}$ satisfies the Mittag-Leffler condition and hence the $\varprojlim ^1$-term vanishes. 
\end{proof}

\begin{cor}
Let $A$ be a separable $\sigma$-$\Cst$-algebra and let $\beta _t$ be a homotopy of continuous actions of a compact Lie group $G$ on a $\sigma$-$\Cst$-algebra $B$. We write $B_t$ for $\sigma $-$G$-$\Cst$-algebras $(B,\beta _t)$. If $\KK^G_*(A,B_0)$ and $\KK^G_*(A,B_1)$ are finitely generated for $\ast=0,1$, there is an isomorphism
$$\KK ^G(A,B_0) _{I_G^\F}^{\myhat} \to \KK^G(A,B_1)_{I_G^\F}^{\myhat}.$$
\end{cor}

\section{Restriction map in the Kasparov category}\label{section:4}
The main subjects in this section are the families $\FIN$ of all finite subgroups and $\FZ$ of all finite cyclic subgroups of $G$. We revisit McClure's restriction map theorem (Theorem A and Corollary C of \cite{MR862427}) and its generalization for $\KK$-theory by Uuye (Theorem 0.1 of \cite{MR2887199}) from categorical viewpoint. 

First, we introduce the Kasparov category with coefficient. Let $M_n$ be the mapping cone of the $n$-fold covering map $S^1 \to S^1$ and $M:=\Tel M_n$ where $M_n \to M_{nm}$ is induced from the $m$-fold covering of $S^1$. Set $Q_n:=C_0(M_n)$, $Q:=\hoprojlim Q_n=\{ f \in C(M) \mid f(\ast)=0 \}$ and 
\ma{
\KK ^G(A,B;\zahl /n\zahl )&:=\KK^G(A,B \otimes Q_n),\\
\KK ^G(A,B;\hat{\zahl})&:=\KK^G(A,B \otimes Q)}
(when we consider the real $\KK$-groups, we replace $M_n$ with $S^6M_n$). The Kasparov category $\sigma \Kas ^G_{\hat{\zahl}}$ whose morphism set is $\KK$-groups with coefficient in $\hat{\zahl}$ has a canonical structure of the triangulated category.  
Actually, it follows in the same way as Theorem \ref{thm:homideal} that the exact triangle 
\[A \otimes SQ \to A \otimes R \to A \to A \otimes Q, \]
where $R$ is the mapping cone of $\comp \to Q$, gives the decomposition of $A$ with respect to the complementary pair $(\Ncat_\J, \ebk{\Icat _\J}^\loc)$ for the family of homological ideals $\J_n:=\Ker (\blank \otimes Q_n)$.

\begin{lem}\label{lem:coeff}
Let $A$ and $B$ be $\sigma$-$G$-$\Cst$-algebras. We have the exact sequence
$$0 \to \frac{\KK^G_*(A,B)}{n\KK^G_*(A,B)} \to \KK^G_*(A,B;\zahl /n \zahl) \to \Tor _1^\zahl (\zahl /n \zahl , \KK^G_{*+1}(A,B)) \to 0.$$
Moreover, if $\KK^G_*(A,B)$ are finitely generated for $\ast =0,1$, we have
$$\KK^G_*(A,B;\hat{\zahl}) \cong  \varprojlim \KK^G(A,B;\zahl/n\zahl) \cong \KK^G(A,B)^{\myhat}.$$
\end{lem}
\begin{proof}
It is proved in the same way as Theorem 2.7 of \cite{MR880503}. Consider the six-term exact sequence
\[
\xymatrix{
\KK^G(A,B) \ar[r]^{n\cdot } &\KK^G(A,B) \ar[r] & \KK^G(A,B;\zahl/n\zahl) \ar[d] \\
\KK^G_1(A,B;\zahl /n\zahl) \ar[u] & \KK^G_1(A,B) \ar[l] & \KK^G_1(A,B) \ar[l]^{n\cdot}
}
\]
induced from the exact triangle $S^2 \to Q_n \to S \xra{n \cdot } S$. Then we get the first exact sequence since $\Tor _1^\zahl (\zahl /n \zahl , \KK^G_{*+1}(A,B))$ is equal to the kernel of multiplication by $n$. The second exact sequence is obtained as the projective limit of the first one. Note that any finitely generated $R(G)$-module does not contain a divisible subgroup.
\end{proof}

Let $G$ be a compact Lie group. Let $T$ be a closed subgroup of $G$ as in Proposition 4.1 of \cite{MR0248277}, that is, it is isomorphic to a finite extension of a torus and $\comp$ is a direct summand of $C(G/T)$ in $\sigma \Kas ^G$. According to Corollary 1.2 of \cite{MR880503}, there is an increasing sequence $\{ F_n \}$ of finite subgroups of $T$ such that $\pi (F_i)=T/T^0$ and any cyclic subgroup of $T^0$ is contained in $T^0 \cap F_i$ for sufficiently large $i>0$. 

\begin{lem}\label{lem:tori}
Set $\Phi:=\Tel T/F_i$. Then, $C(\Phi) \otimes Q_n \sim _{\KK^G} C(E_\FIN T) \otimes Q_n$ and $C(\Phi) \otimes Q \sim _{\KK^G} C(E_\FIN T) \otimes Q$. In particular, $C(\Phi)$ is equivalent to $C(E_\FIN T)$ in the category $\sigma \Kas ^T _{\hat{\zahl}}$. 
\end{lem}
In other words, $SC(\Phi) \to \cone _\Phi \to \comp \to C(\Phi)$ is the approximation triangle of $\comp$ with respect to the complementary pair $(\mathcal{EC}, \ebk{\mathcal{EI}}^\loc)$ in $\sigma \Kas ^T_{\hat{\zahl}}$. 
\begin{proof}
By Corollary 1.4 and Corollary 1.5 of \cite{MR862427}, $\Phi \times E_\F T \times M_n$  (resp. $\Phi \times E_\FIN T \times M$) is $G$-homotopic to $E_\FIN T \times M_n$ (resp. $E_\FIN T \times M$). Since $C(\Phi)$ is in $\ebk{\FIN \mathcal{I}}^\loc$, we obtain $\KK^G$-equivalences
$$C(\Phi) \otimes A \sim _{\KK^G} C(\Phi) \otimes C(E_\FIN T) \otimes A \sim _{\KK^G} C(E_\FIN T) \otimes A$$
for $A=Q_n$ or $Q$.
\end{proof}

\begin{prp}\label{prp:Z/nZ}
Let $A,B$ be separable $\sigma$-$G$-$\Cst$-algebras. Then we have an equivalence of filtrations 
$$(\J _G^\FIN)^*(A,B;\zahl /n\zahl ) \sim \{ \J _G^{F_n}(A,B;\zahl /n\zahl ) \} _{n \in \zahl _{>0}}$$
where $\J(A,B ; \zahl /n\zahl):=\J(A,B \otimes Q_n)$ for any homological ideals $\J$.
\end{prp}
\begin{proof}
By Corollary \ref{cor:cyc}, the filtration $(\J_G^\FIN)^*$ is equivalent to the restriction of $(\J_T^\F)^*$ of $\KK^T_*(A,B)$ onto the direct summand $\KK^G_*(A,B)$ where $\F$ is the smallest family of $T$ containing $F_n$'s. Hence, it suffices to prove the assertion for $T$. Let $Y_k:=\bigast _{i=1}^k T/F_{n_i}$ be the $k$-th step of the Milnor construction and $Y:=\Tel Y_k$. By Lemma \ref{lem:tori} and Corollary \ref{cor:sigma}, we obtain the pro-isomorphism between projective systems $\{ C(Y_k) \otimes Q_n \}$  and $\{ C(T/F_k) \otimes Q_n \}$. Therefore we obtain the equivalences 
$$\{ \KK^T_*(A,B \otimes C(Y_k);\zahl /n\zahl ) \}_k \to \{ \KK^T_*(A,B \otimes C(T/T_k) ;\zahl /n\zahl )\}_k$$
of projective systems of $R(G)$-modules. Finally we get the equivalence of filtrations given by the kernels of canonical homomorphisms from $\KK^T(A,B;\zahl/n\zahl)$ to them, which is the conclusion.
\end{proof}

\begin{lem}\label{lem:elem}
Let $\FINe$ be the family of all elementary finite subgroups of $G$. Then, filtrations $(\J_G^{\FIN})^*$ and $(\J_G^{\FINe})^*$ are equivalent.
\end{lem}
\begin{proof}
Let $H$ be a finite subgroup of $G$. For an inclusion of finite groups $L \leq H$, $i_L^H : \KK^L(A,B) \to \KK^H(A,B)$ denotes the $\ell^2$-induction functor $\Ind _L^H (\blank) \otimes _{c(H/L)}[\ell ^2 (H/L)]$. By Brauer's induction theorem, $1 \in R(H)$ is of the form $\sum _j i_{L_j}^H \xi _j$ where $L_j$'s are elementary finite subgroups of $H$ and $\xi_j \in R(L_j)$. Then, we have
$$\Res _G^Hx =\sum _j i_{L_j}^H (\xi _j) \Res _G^Hx = \sum _j  i_{L_j}^H (\xi _j \cdot \Res _G^{L_j}x)$$
for any $x \in \KK^G(A,B)$. Consequently, $\Res _G^Hx=0$ for any $H \in \FIN$ if and only if $\Res _G^L=0$ for any $L \in \FINe$.
\end{proof}

\begin{thm}[cf.\ Theorem 0.1 of \cite{MR2887199}]\label{thm:rest}
Let $G$ be a compact Lie group and let $A$ and $B$ separable $G$-$\Cst$-algebras. We assume that $\KK^G_*(A,B)$ are finitely generated for $\ast=0,1$. Then the following hold:
\begin{enumerate}
\item If $\KK^H(A,B)=0$ holds for any finite cyclic subgroup $H$ of $G$, then $\KK^G(A,B)=0$.
\item If $\xi \in \KK^G(A,B)$ satisfies $\Res_G^Hx=0$ for any elementary finite subgroup $H$ of $G$, then $x=0$.
\end{enumerate}
\end{thm}
Note that it is assumed in \cite{MR2887199} that $\KK^H(A,B)$ are finitely generated $R(G)$-modules for any closed subgroup $H \leq G$. 
\begin{proof}
Consider the homological functor $\KK^G(A, B \otimes \blank)$. Since the full subcategory of all $F$-contractible objects is colocalizing and contains all $\Cst$-algebras of the form $C(G/H) $ for $H \in \F$, we have $\KK^G(A,\tilde{B})=\KK^G(A, B \otimes C(E_\F G))=0$. Now (1) follows from Theorem \ref{thm:AS} and Corollary 3.3 of \cite{MR862427}.

Next we show (2). Let $\xi \in \KK^G(A,B)$ such that $\Res _G^H\xi =0$ for any $H \in \FINe$ and write $\xi_n \in \KK^G(A,B;\zahl /n\zahl)$ and $\hat{\xi} \in \KK^G(A,B ;\hat{\zahl})$ for the corresponding elements. By Lemma \ref{lem:coeff}, $\KK^G_*(A,B;\zahl /n\zahl)$ are also finitely generated $R(G)$-modules. Hence, by Theorem \ref{thm:AS}, Proposition \ref{prp:Z/nZ}, Lemma \ref{lem:elem} and Corollary 3.3 of \cite{MR862427}, $\Res _G^H \xi_n=0$ for all $H \in \FINe$ implies $\xi_n=0$. Using Lemma \ref{lem:coeff} and Corollary 3.3 of \cite{MR862427} again, we obtain $\hat{\xi}=0$ and hence $\xi=0$.
\end{proof}

\section{Generalization for groupoids and proper actions}\label{section:5}
In this section, we generalize the Atiyah-Segal completion theorem for equivariant $\KK$-theory of certain proper topological groupoids. Groupoid equivariant $\K$-theory and $\KK$-theory are studied, for example, in \cite{MR1686846} and \cite{MR1671260}. 

First, we recall some conventions on topological groupoids. Let $\G =(\G ^1 , \G ^0 , s,r)$ be a second countable locally compact Hausdorff topological groupoid with a haar system. 
We assume that $\G$ is proper, that is, the combination of the source and the range maps $(s,r):\G^1 \to \G^0 \times \G^0$ is proper. We write $[\G]$ for the orbit space $\G ^0 /\G$ of $\G$ and $\pi : \G ^0 \to [\G]$ for the canonical projection. For a closed subset $S \subset \G ^0$, let $\G_S$ denote the full subgroupoid given by $\G_S^1:=\{ g \in \G^1 \mid s(g), r(g) \in S \}$ and $\G_S^0:=S$. 

Hereafter we deal with proper groupoids satisfying the following two conditions.
\maa{
\text{\parbox{.85\textwidth}{For any $x \in \G^0$, there is an open neighborhood $U$ of $x$, a compact $\G _x^x$-space $S_x$ with a $\G _x^x$-fixed base point $x_0$ and a groupoid homomorphism $\varphi _x:\G_x^x \ltimes S_x \to \G_{\overline{U}}$ such that
\begin{itemize}
\item the inclusion $\{ x_0 \} \to S_x$ is a $\G _x^x$-homotopy equivalence,
\item the homomorphism $\varphi _x $ is injective and a local equivalence (Definition A.4 of \cite{MR2860342}) such that $\varphi _x(x_0)=x$ and $\varphi _x |_{\G _x^x \ltimes \{ x_0 \}} =\id _{\G_x^x} $,
\end{itemize}
}}\label{cond:gr1}\\
\text{\parbox{.85\textwidth}{The groupoid $\G$ admits a finite dimensional unitary representation whose restriction on $\G_x^x$ is faithful for each $x \in \G^0$.}\label{cond:gr2}
}}
We say that a triple $(\overline{U}, S_x, \varphi _x)$ as in (\ref{cond:gr1}) is a slice of $\G$ at $x$.
\begin{exmp}
The slice theorem for $G$-CW-complexes (Theorem 7.1 of \cite{LcukUribe}, see also Lemma 4.4 (ii)) implies that (\ref{cond:gr1}) holds for $\G$ such that for any $x \in \G$ there is a saturated neighborhood $U$ of $x$ and a local equivalence $G \ltimes X \to \G _U$ where $G$ are Lie groups and $X$ are $G$-CW-complexes. 
\end{exmp}

\begin{exmp}
All proper Lie groupoid satisfies (\ref{cond:gr1}). Actually, the slice theorem for proper Lie groupoids (Theorem 4.1 of \cite{MR2292634}) implies that for any orbit $\mathcal{O}$ of $\G$ there isa  tubular neighborhoof $U$ of $\mathcal{O}$ and a local equivalence $\G_{x}^x \ltimes N_x\mathcal{O} \to \G_U$ where $x \in \mathcal{O}$ and $N\mathcal{O}$ is the normal bundle of $\mathcal{O}$. On the other hand, a proper Lie groupoid does not satisfies (\ref{cond:gr2}) in general even if it is an action groupoid. Actually, let $G$ be the group as Section 5 of \cite{MR1838997}. Then, the groupoid $\G:=G \ltimes \real$ is actually a counterexample. To see this, compare Lemma \ref{lem:grcond} (2) below with the fact that $\Im (R(\G) \to R(\G _x^x) \cong R(T))=R(T/K)$ (see p.615 of \cite{MR1838997}).
\end{exmp}

\begin{exmp}
By Lemma \ref{lem:grcond} below and Theorem 6.15 of \cite{MR2499440}, an action groupoid $G \ltimes X$ satisfies (\ref{cond:gr2}) if 
\begin{itemize}
\item $G$ is a closed subgroup of an almost connected group $H$ or
\item $G$ is discrete, $X/G$ has finite covering dimension and all finite subgroups of $G$ have order at most $N$ for some $N \in \zahl _{>0}$.
\end{itemize}
\end{exmp}

\begin{lem}\label{lem:grcond}
Let $\G$ be a proper groupoid whose orbit space is compact.
\begin{enumerate}
\item If the Hilbert $\G$-bundle $L^2\G$ is AFGP (Definition 5.14 of \cite{MR2119241}), then $\G$ satisfies (\ref{cond:gr2}).
\item If $\G$ satisfies (\ref{cond:gr2}), the representation ring $R(\G_x^x)$ is a noetherian module over $R(\G):=\KK ^\G(\comp,\comp)$ for any $x \in \G^0$.
\item If $\G$ satisfies (\ref{cond:gr1}) and (\ref{cond:gr2}), then $R(\G) $ is a noetherian ring.
\end{enumerate}
\end{lem}
\begin{proof}
First we check (1). Let $(\Hilb_n , \pi _n)$ be an increasing sequence of finite dimensional subrepresentations of $L^2\G$ whose union is dense. For any $x \in \G^0$, there is $n>0$ such that $\pi _n |_{\G_x^x}$ is faithful. By continuity, there is a saturated neighborhood $U$ of $x$ such that $\pi _n|_{\G_y^y}$ is faithful for any $y \in U$. We obtain the conclusion since $[\G]$ is compact.

To see (2), take an $n$-dimensional unitary representation $\Hilb$ of $\G$ and let $U(\Hilb)$ be the corresponding principal $U(n)$-bundle. Then we have the ring homomorphism
$$R(U(n)) \to R(\G); \ [V] \mapsto [U(\Hilb) \times _{U(n)}V].$$
Now, the composition $R(U(n)) \to R(\G) \to R(\G _x^x)$ is actually induced from a group homomorphism $\G _x^x \to U(n)$ which is injective by assumption. By Proposition 3.2 of \cite{MR0248277}, $R(\G _x^x)$ is a finitely generated (and hence noetherian) module over $R(U(n))$. Consequently, we obtain that $R(\G _x^x)$ is noetherian as an $R(\G)$-module. 

If $\G$ satisfies (\ref{cond:gr1}) in addition, there is an open covering $\{ U_i \}$ and $x_i \in U_i$ such that $R(\G _{\overline{U}_i})$ is isomorphic to $R(\G_{x_i}^{x_i})$ and in particular is a noetherian $R(\G)$-module. By a Mayer-Vietoris argument, we obtain that $R(\G)$ itself is a noetherian $R(\G)$-module.
\end{proof}

The induction for groupoid $\Cst$-algebras is given in Definition 4.18 of \cite{MR2570950}. Let $\G$ be a second countable locally compact groupoid and $\mathcal{H}$ be a subgroupoid. Let $(\Omega , \sigma ,\rho)$ be a Hilsum-Skandalis morphism \cite{MR925720} from $\G$ to $\mathcal{H}$ given by
$$\Omega :=\{ g\in \G^1 \mid s(g) \in \mathcal{H} ^0 \}, \ \sigma :=s :\Omega \to \mathcal{H} ^0, \ \rho :=r : \Omega \to \G ^0$$
together with the left $\G$-action and the right $\mathcal{H}$-action given by the composition. The induction functor $\sigma \GCsep{\mathcal{H}} \to \sigma \GCsep{\G}$ is given by 
$$\Ind _\mathcal{H} ^\G A=\Omega ^*A:=(C_b (\Omega) \otimes _{\mathcal{H}^0} A)^\mathcal{H}.$$
In the same way as the case of groups, it induces the functor between Kasparov categories.

\begin{prp}
Let $\G$ be a proper groupoid and let $\mathcal{H}$ be a closed subgroupoid. Then, the induction functor $\Ind _\mathcal{H}^\G$ is the right adjoint of the restriction functor $\Res _\G ^\mathcal{H}$, that is,
$$\KK^\G(A,\Ind _\mathcal{H}^\G B) \cong \KK^\mathcal{H}(\Res _\G^\mathcal{H} A,B).$$
\end{prp}
\begin{proof}
We have the isomorphism 
$$\Ind _{\mathcal{H}}^\G \Res _\G ^{\mathcal{H}} A=(C_b(\Omega) \otimes _{\mathcal{H}^0}A)^{\mathcal{H}} \xra{\cong} C_b(\Omega /\mathcal{H}) \otimes _{\G ^0} A; \ a(\gamma) \mapsto \alpha _{\gamma ^{-1}}(a(\gamma)).$$
Let $\Delta$ be the subspace of $\Omega$ consisting of all identity morphisms in $\mathcal{H}$. The same argument as Proposition \ref{prp:indres} we can observe that the following $\ast$-homomorphisms
\ma{
\varepsilon _A :\Res _\G^\mathcal{H} \Ind _\mathcal{H}^\G A \cong (C(\Omega)\otimes _X A)^\mathcal{H} \to A;&\  f \mapsto f|_{\Delta}\\
\eta _B: B \to \Ind _\mathcal{H}^\G \Res _\G^\mathcal{H} B \cong C(\Omega /\mathcal{H})\otimes _X B; &\  a \mapsto a \otimes 1_{\Omega /\mathcal{H}}
}
gives the unit and counit of the adjunctions.
\end{proof}

Now we introduce two generalizations of Theorem \ref{thm:AS}. First we consider a proper groupoid $\G$ satisfying (\ref{cond:gr1}) and (\ref{cond:gr2}). For simplicity, we assume that $[\G]$ is connected. Then we have a ring homomorphism $\dim :R(\G) \to \zahl$. Set $I_\G := \Ker \dim $ be the augmentation ideal. We regard a closed subspace $S \subset \G ^0$ as a subgroupoid consisting of all identity morphisms on $x \in S$. We write $\J _\G^S$ for the homological ideal $\Ker \Res _\G^S$ of $\sigma \Kas ^\G$ and in particular set $\J_\G:=\J_\G^{\G^0}$.  We say that $\sigma$-$\G$-$\Cst$-algebras of the form $A=\Ind _{\G ^0}^\G A_0$ is trivially induced and we write $\mathcal{TI}$ for the class of trivially induced objects. Similarly, we say that $\sigma$-$\G$-$\Cst$-algebras $B$ such that $\Res _\G^{\G^0} B$ is $\KK^{\G ^0}$-contractible is trivially contractible and we write $\mathcal{TC}$ for the class of trivially contractible objects. 

\begin{lem}\label{lem:saturation}
Let $(\overline{U}, S,\varphi )$ be a slice of $\G$ at $x \in \G^0$ and let $V$ be the smallest saturated closed subspace of $\G^0$ containing $\varphi (S)$. 
\begin{enumerate}
\item Let $A$ be a $\sigma$-$\G$-$\Cst$-algebra. If $\Res _\G^S A$ is $\KK^S$-contractible, then $\Res _\G^V A$ is $\KK^V$-contractible. 
\item If $V$ is compact, the filtrations $\J_{\G_S}^*$ and $\J_{\G_V}^*$ are equivalent under the isomorphism $\sigma \Kas ^{\G_S} \cong \sigma \Kas ^{\G_V}$.
\end{enumerate}
\end{lem}
\begin{proof}
Since the homomorphism $\varphi :\G _x^x \ltimes S \to \G$ is a local equivalence, for any $y \in \G_V^0$ we have a closed subspace $W$ of $\G_V^0$ containing $y$ in its interior and a continuous map $f : W  \to \G^1$ such that $s \circ f=\id$ and $r \circ f(W) \subset S$, which induces a group homomorphism 
$$\{ \Ad f(w) \} _{w \in W} : \KK^S(\Res _\G^S A, \Res _\G^S B) \to \KK^W (\Res _\G ^W A,\Res _\G ^W B).$$
Since $\Res _{\G_V}^{W}=\Ad f(u) \circ \Res _{\G_V}^{S}$, we obtain $\J _{\G_V}^W \subset \J_{\G_V}^S$. 

In particular, if $\Res _\G ^S A$ is $\KK^S$-contractible, then $\Res _\G ^WA$ is $\KK^W$-contractible. We obtain (1) because any locaally contractible $X$-$\Cst$-algebra is globally contractible (which follows from a Mayer-Vietoris argument). 

To see (2), let $\{ W_i \}$ be a finite family of closed subspaces of $\G_V^0$ obtained as above such that $\bigcup W_i=\G_V^0$. Then, in the same way as Lemma \ref{lem:local}, we obtain $(\J _{\G_V}^{S})^n \subset \J_{\G_V}^{W_1} \circ \cdots \J_{\G_V}^{W_n} \subset \J _{\G _V}^V$. 
\end{proof}

Consider the following assumption for a pair $(A,B)$ of $\sigma$-$\G$-$\Cst$-algebras corresponding to the assumption that $\KK^G_*(A,B)$ are finitely generated $R(G)$-modules in Theorem \ref{thm:AS}:
\maa{\text{\parbox{.90\textwidth}{There is a basis $\{ U_i \} $ of the topology of $\G$ such that $R(\G)$-modules $\KK^{\G _{\overline{U}_i}}_*(\Res _\G ^{\G_{\overline{U}_i}},\Res _\G ^{\G_{\overline{U}_i}} B)$ are finitely generated. }}\label{cond:gr3}}

\begin{thm}\label{thm:groupoid}
Let $\G$ be a proper groupoid satisfying (\ref{cond:gr1}) and (\ref{cond:gr2}) whose orbit space is compact. Then the following holds:
\begin{enumerate}
\item A pair $(\mathcal{TC},\ebk{\mathcal{TI}}^{\loc})$ is complementary in $\sigma \Kas^\G$.
\item For any pair of $\sigma$-$G$-$\Cst$-algebras $(A,B)$ satisfying (\ref{cond:gr3}), there are isomorphisms of $R(\G)$-modules
$$\KK^\G(A,B)_{I_\G}^{\myhat} \cong \KK^\G (A,\tilde{B}) \cong \RKK^\G(E\G; A,B) \cong \sigma \Kas ^\G /\mathcal{TC} (A,B).$$
\end{enumerate}
\end{thm}
\begin{proof}
The assertion (1) can be shown in the same way as Theorem \ref{thm:decomp}.

To see (2), take slices $\{ (X_i, S_i, \varphi _i) \}_{i \in I}$ such that $\KK^{\G_{X_i}}_*(\Res _{\G}^{\G_{X_i}}A,\Res _{\G}^{\G_{X_i}}B)$ are finitely generated and $\bigcup \pi (X_i)=[\G]$. Consider the groupoid 
$$\tilde{\G}^0 :=\bigsqcup S_i, \ \tilde{\G}^1:=\{ (g,i,j) \in \G \times I \times I \mid s(g) \in \varphi _i(S_i),r(g) \in \varphi _j(S_j) \}$$
with $s(g,i,j)=s(g) \in S_i$, $r(g,i,j)=r(g) \in S_j$ and $(h,j,k) \circ (g,i,j)=(g\circ h,i,k)$. Then, $\tilde{\G}$ is Morita equivalent to $\G$ and we have the family of closed full subgroupoids $\{ \G_i :=\G |_{\pi ^{-1}(\pi (S_i))} \}_{i \in I}$ such that $\tilde{\G} = \bigcup \G_i$ and the pair $(A|_{\G_i^0},B|_{\G_i^0})$ of $\sigma$-$\G_i$-$\Cst$-algebras satisfies (2). 

Let $\mathcal{H}$ be a proper groupoid which admits a local equivalence $\varphi : G \ltimes X \to \mathcal{H}$ where $G$ is a compact Lie group and $X$ is a compact $G$-CW-complex (such as $\G_i$ or $\G_i \cap \G_j$). Then, by Lemma \ref{lem:saturation}, $I_\mathcal{H}$-adic topology and $I_G$-adic topology on $\KK^\mathcal{H}(A,B) \cong \KK^{G \ltimes X}(\varphi ^*A,\varphi ^*B)$ coincide. Moreover $\varphi ^*$ preserves $\mathcal{TC}$ and $\ebk{\mathcal{TI}}^\loc$. Hence, (2) holds for $\mathcal{H}$ by Theorem \ref{thm:AS}. 

By Lemma 3.4 of \cite{MR935523} and the proof of Lemma \ref{lem:grcond}, $I_\G$-adic and $I_{\G _i}$-adic topologies coincide on $\KK^{\G_i}(A|_{\G_i^0},B|_{\G_i^0})$. Moreover, $\Res _\G ^{\G_i}$ preserves $\mathcal{TC}$ and $\ebk{\mathcal{TI}}^\loc$. Finally we obtain (2) for $\tilde{\G}$ by using the Mayer-Vietoris exact sequence
\[
\xymatrix@C=1.5em{
\scriptstyle \cdots \ar[r] &\scriptstyle  \KK^{\G }(A,B)_{I_\G}^{\myhat} \ar[r] \ar[d] &  {\begin{array}{c} \scriptstyle \KK^{\G _1}(\Res_\G^{\G_1}A,\Res _\G^{\G_1}B)_{I_\G}^{\myhat}\\ \scriptstyle \oplus \\ \scriptstyle \KK^{\G _2}(\Res _\G^{\G_2}A,\Res _\G^{\G_2}B)_{I_\G}^{\myhat} \end{array}} \ar[r] \ar[d]&\scriptstyle   \KK^{\G _0}(\Res_\G^{\G_0}A,\Res _\G^{\G_0}B)_{I_\G}^{\myhat} \ar[r] \ar[d] & \scriptstyle \cdots\\
\scriptstyle \cdots \ar[r] & \scriptstyle  \KK^{\G }(A,\tilde{B}) \ar[r] & {\begin{array}{c} \scriptstyle  \KK^{\G _1}(\Res _\G^{\G_1}A,\Res_\G^{\G_1}\tilde{B}) \\ \scriptstyle   \oplus \\ \scriptstyle \KK^{\G _2}(\Res _\G^{\G_2},\Res_\G^{\G_2}\tilde{B}) \end{array}} \ar[r] &\scriptstyle  \KK^{\G _0}(\Res_\G^{\G_0}A,\Res_\G^{\G_0}\tilde{B})\ar[r] & \scriptstyle \cdots
}
\]
(for $\G =\G_1 \cup \G_2$, $\G_0:=\G_1 \cap \G_2$) and the five lemma recursively. Note that the first row is exact because the completion functor is exact when modules are finitely generated. Since the augmentation ideal $I_\G$ and the complementary pair $(\mathcal{TC},\ebk{\mathcal{TI}}^\loc)$ are preserved under Morita equivalence, we obtain the consequence.
\end{proof}

Second generalization is the Atiyah-Segal completion theorem for proper actions. Let $G$ be one of 
\begin{itemize}
\item a countable discrete group such that and all finite subgroups of $G$ have order at most $N$ for some $N \in \nat$ and has a model of the universal proper $G$-space $E_\mathcal{C} G$ which is $G$-compact and finite covering dimension or
\item a cocompact subgroup of an almost connected second countable group
\end{itemize}
and let $\F$ be a family of $G$ consisting of compact subgroups. Set $\G:=G \ltimes E_\mathcal{C}G$. According to Section 7 of \cite{MR2193334}, the category $\sigma \Kas ^\G$ is identified with the subcategory $\ebk{\CI}_\loc$ of $\sigma \Kas ^G$ by the natural isomoprhism 
$$p_{E_\mathcal{C}G}^*:\KK^G(A,B) \xra{\cong} \KK^{G \ltimes E_\mathcal{C}G}(A \otimes C(E_\mathcal{C}G),B \otimes E_\mathcal{C}G)$$
since $G$ has a Dirac element coming from a proper $\sigma$-$G$-$\Cst$-algebra when $G$ is discrete (Theorem 2.1 of \cite{MR2105483}) or a closed subgroup of an almost connected second countable group $H$ (Theorem 4.8 of \cite{MR918241}).

\begin{thm}\label{thm:proper}
Let $G$ and $\F$ be as above. Then, the following holds:
\begin{enumerate}
\item A pair $(\FC,\ebk{\FI}^{\loc})$ is complementary in $\ebk{\CI}_\loc \subset \sigma \Kas ^G$.
\item For any pair of proper $\sigma$-$G$-$\Cst$-algebras $A$, $B$ such that $\KK^H_*(A,B)$ are finitely generated for any compact subgroup $H$ of $G$, there are isomorphisms of $R(\G)$-modules
$$\KK^G(A,B)_{I_\G^\F}^{\myhat} \cong \KK^G (A,\tilde{B}) \cong \RKK^G(E_\F \G; A,B) \cong \sigma \Kas ^G /\mathcal{FC} (A,B).$$
\end{enumerate}
\end{thm}
\begin{proof}
The proof is given in the same way as Theorem \ref{thm:groupoid}. Note that $\J _\G^H = \J _\G ^{H \ltimes X}$ for any $H$-subspace $X$ of $E_\mathcal{C}G$ (even if $X$ is not compact) since the composition
$$\sigma \Kas^{H \ltimes E_\mathcal{C}G} \to \sigma \Kas^{H \ltimes X} \to \sigma \Kas ^{H}$$
is identity.
\end{proof}

\section{The Baum-Connes conjecture for group extensions}\label{section:6}
In this section we apply Corollary \ref{cor:cyc} for the study of the complementary pair $(\ebk{\CI}_\loc ,\CC)$ of the Kasparov category $\sigma \Kas ^G$ when $G$ is a Lie group. As a consequence, we refine the theory of Chabert, Echterhoff and Oyono-Oyono \citelist{\cite{MR1817505},\cite{MR1857079},\cite{MR1836047}} on permanence property of the Baum-Connes conjecture under extensions of groups. 

Let $G$ be a second countable locally compact group such that any compact subgroup of $G$ is a Lie group. We bear the case that $G$ is a real Lie group in mind. We write $\mathcal{C}$ and $\CZ$ for the family of compact and compact cyclic subgroups of $G$ respectively. 
\begin{cor}\label{cor:BCcyc}
We have $\CC = \CZC$ and $\ebk{\CI}_\loc =\ebk{\CZI}_\loc$. 
\end{cor}
\begin{proof}
Since $\CZ \subset \mathcal{C}$, we have $\CZI \subset \CI$ and $\CC \subset \CZC$. Hence it suffices to show $\CC=\CZC$, which immediately follows from Corollary \ref{cor:cyc} (2).
\end{proof}

\begin{cor}[cf.\ Theorem 1.1 of \cite{MR2083579}]
The canonical map $f:E_{\mathcal{CZ}}G \to E_{\mathcal{C}}G$ induces the $\KK^G$-equivalence $f^*:C(E_{\mathcal{CZ}}G)\to C(E_{\mathcal{C}}G)$.
\end{cor}
Note that the topological $\K$-homology group $\K ^{\mathrm{top}}_*(G;A)$ is isomorphic to the $\KK$-group $\KK^G(C(E_{\mathcal{C}}G),A)$ of $\sigma$-$\Cst$-algebras for any $G$-$\Cst$-algebra $A$. 
\begin{proof}
Since $f$ is a $T$-equivariant homotopy equivalence between $E_\mathcal{C}G$ and $E_\CZ G$ for any $T \in \CZ$, $f^*$ is an equivalence in $\sigma \Kas^G /\CZC$. The conclusion follows from Corollary \ref{cor:BCcyc} because $C(E_\CZ G)$ and $C(E_\mathcal{C}G)$ are in $\ebk{\CI}_\loc =\ebk{\CZI}_\loc$. 
\end{proof}

Next we review the Baum-Connes conjecture for extensions of groups. Let $1 \to N \to G \to Q \to 1$ be an extension of second countable locally compact groups. We assume that any compact subgroup of $Q$ is a Lie group. As in Subsection 5.2 of  \cite{MR2313071}, we say that a subgroup $H$ of $G$ is \textit{$N$-compact} if $\pi (H)$ is compact in $Q$. We write $\mathcal{C}_N$ for the family of $N$-compact subgroups of $G$. Then, we have the complementary pair $(\left< \CNI \right>_{\loc}, \CNC)$. It is checked as following. First, in the same way as Lemma 3.3 of \cite{MR2193334}, for a large compact subgroup $H$ of $Q$ we have
$$\KK^G(\Ind _{\tilde{H}}^G A,B) \cong \KK^{\tilde{H}}(\Res _{\tilde{U}_H}^{\tilde{H}} \Ind _{\tilde{H}}^{\tilde{U}_H}A,\Res _G^{\tilde{H}}B)$$
where $\tilde{H}:=\pi ^{-1}(H)$ for any $H \leq Q$ and $U_H$ is as Section 3 of \cite{MR2193334}. Hence $\KK^G(Q,M)=0$ for any $Q \in \CNI$ and $M \in \CNC$. Let $S\mathsf{M} \to \mathsf{Q} \to \comp \to \mathsf{M}$ be the approximation exact triangle of $\comp$ in $\sigma \Kas ^Q$ with respect to $(\ebk{\CI}_\loc,\CC)$. Since the functor $\pi^*:\sigma \Kas ^Q \to \sigma \Kas ^G$ maps $\CI$ to $\CNI$ and $\CC$ to $\CNC$ respectively, $S\pi^*\mathsf{M} \to \pi^*\mathsf{Q} \to \comp \to \pi^*\mathsf{M}$ gives the approximation of $\comp$ in $\sigma \Kas ^G$ with respect to $(\ebk{\CNI}_\loc,\CNC)$. Hereafter, for simplicity of notations we omit $\pi ^*$ for $\sigma$-$Q$-$\Cst$-algebras which are regarded as $\sigma$-$G$-$\Cst$-algebras.

Since $\CI \subset \CNI$ and $\CNC \subset \CC$, we obtain the diagram of semi-orthogonal decompositions
\maa{
\xymatrix{
\ebk{\CI}_\loc \ar@{=}[r] \ar[d]&\ebk{\CI}_\loc \ar[r] \ar[d] &0 \ar[d] \\
\ebk{\CNI}_\loc \ar[r] \ar[d] & \Kas ^G \ar[r] \ar[d] & \CNC \ar@{=}[d] \\
\ebk{\CNI}_\loc \cap \CC \ar[r] & \CC \ar[r] & \CNC,}
\xymatrix{
\mathsf{P} \ar@{=}[r] \ar[d]^{\mathsf{D}_G^Q}& \mathsf{P} \ar[r] \ar[d]^{\mathsf{D}_G} &0 \ar[d] \\
\mathsf{Q} \ar[r]^{\mathsf{D}_Q} \ar[d] &\comp \ar[r] \ar[d] & \mathsf{M} \ar@{=}[d] \\
\mathsf{Q} \otimes \mathsf{N} \ar[r] & \mathsf{N} \ar[r] & \mathsf{M}.} \label{form:PQ}
}

For a $\sigma$-$G$-$\Cst$-algebra $A$, the (full or reduced) crossed product $N \ltimes A$ is a twisted $\sigma$-$Q$-$\Cst$-algebra (Definition 2.1 of \cite{MR1002543}). By the Packer-Raeburn stabilization trick (Theorem 1 of \cite{MR1277761}), it is Morita equivalent to the untwisted $Q$-$\Cst$-algebra 
$$N \ltimes ^{\mathrm{PR}} A:=C_0(Q,N \ltimes A ) \rtimes _{\tilde{\alpha} , \tilde{\tau}} Q$$
where $\tilde{\alpha}$ and $\tilde{\tau}$ are induced from the canonical $G$-action on $C_0(Q,N \ltimes A)$. The \textit{Packer-Raeburn crossed product} $N \ltimes ^{\mathrm{PR}}\blank$ is a functor from $\GCsep{G}$ to $\GCsep{G/N}$, which induces the \textit{partial descent functor}  (Section 4 of \cite{MR1857079})
$$j_G^Q: \sigma \Kas^G \to \sigma \Kas ^{Q}$$
by universality of $\sigma \Kas ^G$ (Theorem \ref{thm:univ}). 
\begin{lem}
The functor $j_G^Q$ maps $\left< \CNI \right>_\loc$ to $\left < \mathcal{CI} \right >_\loc$ and $\CNC$ to $\mathcal{CC}$. 
\end{lem}
\begin{proof}
Let $H$ be a $N$-compact subgroup of $G$ and let $A$ be a $\sigma$-$H$-$\Cst$-algebra. Then, $N \ltimes ^{\mathrm{PR}}\Ind _H^GA$ admits a canonical $\sigma$-$Q \ltimes ((Q \times H\backslash G)/G)$-$\Cst$-algebra structure. Since the $Q$-action on $(Q \times H\backslash G)/G$ is proper, $N \ltimes ^{\mathrm{PR}}\Ind _H^GA$ is in $\ebk{\CI}_\loc$. Consequently we obtain $j_G^Q (\left< \CNI \right>_\loc) \subset \left < \mathcal{CI} \right >_\loc$.

Let $A$ be a $\mathcal{C}_N$-contractible $\sigma$-$\Cst$-algebra. Then, for any compact subgroup $H$ of $Q$, $\Res _Q^H (N \ltimes^{\mathrm{PR}} A)=N \ltimes \Res _G^{\pi^{-1}(K)} A$ is $\KK ^H$-contractible. Hence we obtain $j_G^Q (\CNC) \subset \mathcal{CC}$. 
\end{proof}

Consider the partial assembly map
$$\mu _{G,A}^{Q}: \K^{\mathrm{top}}_*(G; A) \to \K ^{\mathrm{top} }_*(Q;N \ltimes A)$$
constructed in Definition 5.14 of \cite{MR1836047}. Then, in the same way as Theorem 5.2 of \cite{MR2104446}, we have the commutative diagram
\[ 
\xymatrix{
\K^{\mathrm{top}}_*(G;\mathsf{P} \otimes A) \ar[d]^\cong \ar[r]^\cong & \K ^{\mathrm{top} }_*(G;\mathsf{Q} \otimes A) \ar[r]^\cong \ar[d]& \K^{\mathrm{top}}_*(G;A)\ar[d]^{\mu _{Q,A}^Q} \\
\K^{\mathrm{top}}_*(Q;N \ltimes^{\mathrm{PR}} (\mathsf{P} \otimes A)) \ar[r] \ar[d]^\cong & \K^{\mathrm{top}}_*(Q;N \ltimes^{\mathrm{PR}} (\mathsf{Q} \otimes A)) \ar[r]^\cong \ar[d]^\cong& \K^{\mathrm{top}}_*(Q;N \ltimes^{\mathrm{PR}} A)\ar[d]^{\mu _{Q, N \ltimes ^{\mathrm{PR}}A}} \\
\K _*(G \ltimes (\mathsf{P} \otimes A)) \ar[r]^{j_G(\mathsf{D}_G^Q)} & \K _*(G \ltimes (\mathsf{Q} \otimes A)) \ar[r]^{j_G(\mathsf{D}_Q)} & \K_*(G \ltimes A)
}
\]
and hence the composition of partial assembly maps 
$$\mu _{G,A}=\mu _{Q,N \ltimes ^{\mathrm{PR}}A} \circ \mu_{G,A}^Q : \K^{\mathrm{top}}_*(G;A) \to \K^{\mathrm{top}}_*(Q;N \ltimes ^{\mathrm{PR}}A) \to \K _*(G \ltimes A)$$
is isomorphic to the canonical map $\K _*(G \ltimes (\mathsf{P} \otimes A)) \to \K _*(G \ltimes (\mathsf{Q} \otimes A)) \to \K_*(G \ltimes A)$. In other words, the partial assembly map $\mu _{G,A}^Q$ is isomorphic to the assembly map $\mu _{G, \mathsf{Q}\otimes A}$ for $\mathsf{Q} \otimes A$.

We say that a separable $\sigma$-$G$-$\Cst$-algebra $A$ satisfies the (resp.\ strong) Baum-Connes conjecture (BCC) if $j_G(\mathsf{D}_G)$ induces the isomorphism of $\K$-groups (resp.\ the $\KK$-equivalence). 
\begin{thm}\label{thm:BC}
Let $1 \to N \to G \to Q \to 1$ be an extension of second countable groups such that all compact subgroups of $Q$ are Lie groups and let $A$ be a separable $\sigma$-$G$-$\Cst$-algebra. Then the following holds.
\begin{enumerate}
\item If $\pi^{-1}(H)$ satisfies the (resp.\ strong) BCC for $A$ for any $H \in \CZ$, then $G$ satisfies the (resp.\ strong) BCC for $A$ if and only if $Q$ satisfies the (resp.\ strong) BCC for $N \ltimes ^{\mathrm{PR}}_rA$.
\item If $\pi^{-1}(H)$ for any $H \in \CZ$ and $Q$ have the $\gamma$-element, then so does $G$. Moreover, in that case $\gamma _{\pi^{-1}(H)}=1$ for any $H \in \CZ$ and $\gamma _Q=1$ if and only if $\gamma _G=1$.
\end{enumerate}
\end{thm}
\begin{proof}
To see (1), it suffices to show that $G$ satisfies the (resp.\ strong) BCC for $\mathsf{Q} \otimes A$. Consider the full subcategory $\mathfrak{N}$ of $\sigma \Kas ^G$ consisting of objects $D$ such that $G$ satisfies the (resp./ strong) BCC for $D \otimes A$. Set $\CZI_1$ be the family of all $G$-$\Cst$-algebras of the form $C_0(Q/H)$ for $H \in \CZ$. By assumption, $\mathfrak{N}$ contains $\pi ^* \CZI _1$. Since $\mathfrak{N}$ is localizing and colocalizing, $\mathfrak{N}$ contains $\pi^*\ebk{\CZI_1}_\loc^\loc$, which is equal to $\pi^*\ebk{\CI_1}_\loc^\loc$ because $C_0(Q/H)$ are $\KK^G$-equivalent to $C_0(Q/H) \otimes C(E_\CZ H)\in \pi^*\ebk{\CZI _1}^\loc$. By Proposition 9.2 of \cite{MR2193334}, we obtain $\mathsf{Q} \in \mathfrak{N}$. 

The assertion (2) is proved in the same way as Theorem 33 of \cite{MR2313071}. Actually, since  we may assume without loss of generality that $Q$ is totally disconnected by Corollary 34 of \cite{MR2313071}, the homomorphism
$$\mathsf{D}_G ^*:\KK^G(A,\mathsf{P}) \to \KK^G(\mathsf{P} \otimes A , \mathsf{P} )$$
is an isomorphism if $A \in \pi^*\ebk{\CZI}_\loc$ and in particular when $A=\mathsf{Q}$ (note that any compact subgroup is contained in an open compact subgroup which is also a Lie group by assumption). Consequently we obtain a left inverse $\eta_G^Q :\mathsf{Q} \to \mathsf{P}$ of $\mathsf{D}_G^Q$. Now, the composition $\eta _G^Q \circ \pi ^* \eta _Q: \comp \to \mathsf{P}$ is a dual Dirac morphism of $G$. Of course $\eta_G \circ \mathsf{D}_G =\id _\comp$ if $\eta _G^Q \circ \mathsf{D}_G^Q=\id _{\mathsf{Q}}$ and $\eta _Q \circ \mathsf{D}_Q=\id_\comp$.
\end{proof}

\appendix

\section{Equivariant $\KK$-theory for $\sigma$-$\Cst$-algebras}\label{section:App}
In this appendix we summarize basic properties of equivariant $\KK$-theory for $\sigma$-$C^*$-algebras for the convenience of readers. Most of them are obvious generalizations of equivariant $\KK$-theory for $\Cst$-algebras (a basic reference is \cite{MR1656031}) and non-equivariant $\KK$-theory for $\sigma$-$\Cst$-algebras  by Bonkat \cite{Bonkat}. Throughout this section we assume that $G$ is a second countable locally compact topological group.

\subsection{Generalized operator algebras and Hilbert $\Cst$-modules}
Topological properties of inverse limits of $\Cst$-algebras was studied by Phillips in \cite{MR950831}, \cite{MR996445}, \cite{MR1050491} and \cite{MR1050490}. He introduced the notion of representable $\K$-theory for $\sigma$-$\Cst$-algebras in order to formulate the Atiyah-Segal completion theorem for $\Cst$-algebras. 

\begin{defn}
A {\it pro-$G$-$\Cst$-algebra} is a complete locally convex $\ast$-algebra with continuous $G$-action whose topology is determined by its $G$-invariant continuous $\Cst$-seminorms. A pro-$G$-$\Cst$-algebra is a {\it $\sigma $-$G$-$\Cst$-algebra} if its topology is generated by countably many $G$-invariant $\Cst$-seminorms. 
\end{defn}
In other words, pro-$G$-$\Cst$-algebra is projective limit of $G$-$\Cst$-algebras. Actually, a pro-$G$-$\Cst$-algebra $A$ is isomorphic to $\varprojlim _{p \in \mathscr{S}(A)}A_p$ where $\mathscr{S}(A)$ is the net of $G$-invariant continuous seminorms and 
$$A_p:=A/ \{ x \in A \mid p(x^*x)=0 \} $$
is the completion of $A$ by the seminorm $p \in \mathscr{S}(A)$. A pro-$G$-$\Cst$-algebra is {\it separable} if $A_p$ are separable for any $p \in \mathscr{S}(A)$. If $A$ is a separable $\sigma$-$G$-$\Cst$-algebra, then it is separable as a topological space. Basic operations (full and reduced tensor products, free products and crossed products) are also well-defined for pro-$\Cst$-algebras. When $G$ is compact, any $\sigma$-$\Cst$-algebras with continuous $G$-action are actually $\sigma$-$G$-$\Cst$-algebras. 

We write $\sigma \GCsep{G}$ for the category of separable $\sigma$-$G$-$\Cst$-algebras and equivariant $\ast$-homomorphisms. Then we have the category equivalence 
$$\varprojlim : \Pro_{\nat}\GCsep{G} \to \sigma \GCsep{G}$$
where $\Pro_{\nat}\GCsep{G}$ is the category of surjective projective systems of separable $G$-$\Cst$-algebras indexed by $\nat$ with the morphism set $\Hom (\{A_n\},\{B_m\}):=\varprojlim _n \varinjlim _m \Hom (A_n,B_m)$. Actually, a $\ast$-homomorphism $\varphi : A \to B$ induces a morphism between projective systems since the composition $A \xra{\varphi} B \to B_p$ factors through some $A_q$.

Next we introduce the notion of Hilbert module over pro-$\Cst$-algebras. 

\begin{defn}
A \textit{$G$-equivariant pre-Hilbert $B$-module} is a locally convex $B$-module together with the $B$-valued inner product $\ebk{\cdot , \cdot }: E \times E \to B$ and the continuous $G$-action such that $\ebk{e_1,e_2b}=\ebk{e_1,e_2}b$, $\ebk{e_1,e_2}^*=\ebk{e_2,e_1}$, $g(\ebk{e_1,e_2})=\ebk{g(e_1),g(e_2)}$, $g(eb)=g(e)g(b)$ and the topology of $E$ is induced by seminorms $p_E(e):=p(\ebk{e,e})^{1/2}$ for $p \in \mathscr{S}(B)$. A $G$-equivariant pre-Hilbert $B$-module is a \textit{$G$-equivariant Hilbert $B$-module} if it is complete. 
\end{defn}
Basic operations (direct sums, interior and exterior tensor products and crossed products) are also well-defined (see Section 1 of \cite{MR1261582}). 

As a locally convex space, $E$ is isomorphic to the projective limit $\varprojlim _{p \in \mathscr{S}(B)}E_p$ where $E_p:=E/\{ e\in E \mid p(\ebk{e,e})=0 \}$. A $G$-equivariant Hilbert $B$-module $E$ is \textit{countably generated} if $E_p$ are countably generated for any $p \in \mathscr{S}(B)$. 

Let $\mathbb{L}(E)$ and $\Kop (E)$ be the algebra of adjointable bounded  and compact operators on $E$ respectively. They are actually pro-$G$-$\Cst$-algebras since we have isomorphisms 
$$\Lop (E) \cong \varprojlim _{p \in \mathscr{S}(B)} \Lop(E_p), \ \Kop (E) \cong \varprojlim _{p \in \mathscr{S}(B)} \Kop (E_p).$$
In particular, $\mathbb{L}(E)$ and $\Kop (E)$ are $\sigma$-$G$-$\Cst$-algebra if so is $B$. Note that $\mathbb{L}(E)$ is not separable and the canonical $G$-action on $\mathbb{L}(E)$ is not continuous in norm topology. 

Kasparov's stabilization theorem is originally introduced in \cite{MR587371} and generalized by Mingo-Phillips~\cite{MR740176} and Meyer~\cite{MR1803228} for equivariant cases. Bonkat~\cite{Bonkat} also gives a generalization for $\sigma$-$C^*$-algebras. Let $\Hilb$ be a separable infinite dimensional Hilbert space and we write $\Hilb _B$, $\Hilb _{G,B}$ and $\Kop _G$ for $\Hilb \otimes B$, $\Hilb \otimes L^2(G) \otimes B$ and $\Kop (L^2G \otimes \Hilb)$ respectively.
\begin{thm}\label{thm:stab}
Let $B$ be a $\sigma$-unital $\sigma$-$G$-$\Cst$-algebra and let $E$ be a countably generated $G$-equivariant Hilbert $B$-module together with an essential homomorphism $\varphi : \mathbb{K}_G \otimes A \to \Lop (E)$. 
Then there is an isomorphism
$$E \oplus \Hilb_{G,B} \cong \Hilb_{G,B}$$
as $G$-equivariant Hilbert $B$-modules. 
\end{thm}
\begin{proof}
In non-equivariant case, the proof is given in Section 1.3 of \cite{Bonkat}. In fact, we have a sequence $\{ e^i \}$ in $E$ such that $\sup _n \ssbk{e_n^i} \leq 1$ and $\{ \pi (e^i) \}$ generates $E_p$ for any $p \in \mathscr{S}(B)$ since the projection $(E_p)_1 \to (E_q)_1$ between unit balls is surjective for any $p \geq q$. Now we obtain the desired unitary $U$ as the unitary factor in the polar decomposition of the compact operator 
$$T  : \Hilb_B, \to E \oplus \Hilb_B; \ T(\xi ^i)= 2^{-i} e^i \oplus 4^{-i}\xi ^i$$
where $\{ \xi ^i \}$ is a basis of $\Hilb_B$. Actually the range of $|T|$ is dense because $T^*T=\diag (4^{-2},4^{-4},\ldots) + (2^{-i-j}\ebk{e_i,e_j})_{ij}$ is strictly positive.

In equivariant case, we identify $E$ with $L^2(G,A) \otimes_A (L^2(G,A)^* \otimes _{\mathbb{K}_G \otimes A} E)$ and set $E_0:=L^2(G,A)^* \otimes _{\mathbb{K}_GA} E$. Let $U$ be the (possibly non-equivariant) unitary from $\Hilb_B$ to $E_0 \oplus \Hilb_B$ as above. Then we obtain  
$$\tilde{U}(g):=g(U) : C_c(G,\Hilb_B) \to C_c(G,E_0 \oplus \Hilb_B)$$
which extends to the unitary $\tilde {U} : \Hilb_{G,B} \cong L^2(G,\Hilb_B) \to L^2(G,E_0 \oplus \Hilb_B) \cong E \oplus \Hilb _{G,B}$. More detail is found in Section 3 of \cite{MR1803228}.
\end{proof}

A pro-$\Cst$-algebra is \textit{$\sigma$-unital} if there is a strictly positive element $h \in A$. Here, we say that an element $h \in A$ is strictly positive if $\overline{hA}=\overline{Ah}=A$. A pro-$\Cst$-algebra $A$ is $\sigma$-unital if and only if it has a countable approximate unit. A separable $\sigma$-$\Cst$-algebra is $\sigma$-unital and moreover has a countable increasing approximate unit (Lemma 5 of \cite{MR985690}).

\begin{lem}\label{lem:qc}
Let $B$ be a $\sigma$-$\Cst$-algebra with $G$-action, $A \subset B$ a $\sigma$-$G$-$\Cst$-algebra, $Y$ a $\sigma$-compact locally compact space, $\varphi:Y \to B$ a function such that $y \mapsto [\varphi(y),a]$ are continuous functions which take values in $A$. Then there is a countable approximate unit $\{ u_i \}$ for $A$ that is quasi-central for $\varphi(Y)$ and quasi-invariant, that is, the sequences $[u_i,\varphi(y)]$ ($y \in Y$) and $g(u_i)-u_i$ converge to zero.
\end{lem}
\begin{proof}
Let $\{ p_n \} _{n \in \zahl _{>0}}$ be an increasing sequence of invariant $\Cst$-seminorms on $B$ generating the topology of $B$ and let $\{ v_m \}$ be a countable increasing approximate unit for $A$ and $h:=\sum 2^{-k}v_k$. By induction, we can choose an increasing sequence $\{ u_n\}$ given by convex combinations of $v_i$'s such that
\begin{enumerate}
\item $p_n(u_nh-h) \leq 1/n$,
\item $p_n([u_n,\varphi (y)]) \leq 1/n$ for any $y \in \overline{Y_n}$,
\item $p_n(g(u_n)-u_n) \leq 1/n$ for any $g \in \overline{X_n}$.
\end{enumerate}
Each induction step is the same as in Section 1.4 of Kasparov~\cite{MR918241}.
\end{proof}

\begin{thm}\label{thm:tech}
Let $J$ be a $\sigma$-$G$-$\Cst$-algebra, $A_1$ and $A_2$ $\sigma$-unital closed subalgebras of $M(J)$ where $G$ acts continuously on $A_1$, $\Delta$ a separable subset of $M(J)$ such that $[\Delta ,A_1] \subset A_1$ and $\varphi :G \to M(J)$ a function such that $\sup _{g \in G, p \in \mathscr{S}(M(J))}p(\varphi (g))$ is bounded. Moreover we assume that $A_1 \cdot A_2$, $A_1 \cdot \varphi (G)$, and $\varphi (G) \cdot A_1$ are in $J$ and $g \mapsto \varphi(g)a$ are continuous functions on $G$ for any $a \in A_1 +J$. Then, there are $G$-continuous even positive elements $M_1, M_2 \in M(J)$ such that 
\begin{itemize}
\item $M_1+M_2=1$,
\item $M_ia_i$, $[M_i,d]$, $M_2\varphi(g)$, $\varphi (g)M_2$, $g(M_i)-M_i$ are in $J$ for any $a_i \in A_i$, $d \in \Delta$, $g \in G$,
\item $g \mapsto M_2\varphi(g)$ and $g \mapsto \varphi(g)M_2$ are continuous.
\end{itemize}
\end{thm}
\begin{proof}
The proof is given by the combination of arguments in p.151 of \cite{MR918241} and in Theorem 10 of \cite{MR985690}. Actually, by Lemma \ref{lem:qc} we get an approximate unit $\{ u_n \}$ for $A_1$ and $\{ v_n \} $ for $J$ such that
\begin{enumerate}
\item $p_n(u_nh_1-h_1) \leq 2^{-n}$,
\item $p_n([u_n,y]) \leq 2^{-n}$ for any $y \in Y$,
\item $p_n(g(u_n)-u_n) \leq 2^{-n}$ for any $g \in X_n$,
\item $p_n(v_nw-w) \leq 2^{-2n}$ for any $w \in W_n$,
\item $p_n([v_n,z])$ is small enough to $p_n([b_n,z]) \leq 2^{-n}$ for any $z \in \{ h_1, h_2 \} \cup Y \cup \varphi (\overline{X}_n)$,
\item $p_n(g(b_n)-b_n) \leq 2^{-n}$ for any $g \in \overline{X}_n$,
\end{enumerate}
where $h_1$, $h_2$, $k$ are strictly positive element in $A_1$, $A_2$ and $J$ respectively such that $p_n(h_1), p_n(h_2), p_n(k) \leq 1$ for any $n$, $Y \subset \Delta$ is a compact subset whose linear span is dense in $\Delta$, $X_n$ is a increasing sequence of relatively compact open subsets of $G$ whose union is dense in $G$, $W_n:=\{ k,u_nh_2,u_{n+1}h_2 \} \cup u_n \varphi (\overline{X}_n) \cup u_{n+1} \varphi (\overline{X}_{n+1}) \cup \varphi (\overline{X}_n)u_n \cup \varphi (\overline{X}_{n+1})u_{n+1}$ and $b_n:=(v_n-v_{n-1})^{1/2}$. Now, the finite sum $\sum  b_nu_nb_n$ converges strictly to the desired element $M_2 \in M(J)$. 
\end{proof}

\subsection{Equivariant $\KK$-groups}
A generalization of $\KK$-theory for pro-$\Cst$-algebras was first defined by Weidner~\cite{MR1014824} and was generalized for equivariant case by Schochet~\cite{MR1261582}. Here the notion of coherent $A$-$B$ bimodule is introduced in order to avoid Kasparov's technical lemma for pro-$\Cst$-algebras. On the other hand, Bonkat~\cite{Bonkat} introduced a new definition of $\KK$-theory for $\sigma$-$\Cst$-algebras applying the technical lemma \ref{thm:tech} for $\sigma$-$\Cst$-algebras. In this paper we adopt the latter definition.

\begin{defn}
Let $A$ and $B$ be $\sigma$-unital $\zahl /2$-graded $\sigma$-$G$-$\Cst$-algebras. A $G$-equivariant Kasparov $A$-$B$ bimodule is a triplet $(E,\varphi,F)$ where
\begin{itemize}
\item $E$ is a $\zahl /2$-graded countably generated $G$-equivariant Hilbert $B$-module,
\item $\varphi :A \to \mathbb{L}(E)$ is a graded $G$-equivariant $\ast$-homomorphism,
\item $F \in \mathbb{L}(E)_{\mathrm{s.a.}}^{\mathrm{odd}}$ such that $[F,\varphi (A)], \varphi (A) (F^2-1),  \varphi (A)(g(F)-F) \in \Kop(E)$ and $\varphi (a)F,F\varphi (a)$ are $G$-continuous.
\end{itemize}
\end{defn}

Two $G$-equivariant Kasparov $A$-$B$ bimodules $(E_1, \varphi _1 , F_1)$ and $(E_2,\varphi _2, F_2)$ are \textit{unitary equivalent} if there is a unitary $u \in \mathbb{L}(E_1, E_2)$ such that $u\varphi _1u^* =\varphi _2$ and $uF_1u^*=F_2$. Two $G$-equivariant Kasparov $A$-$B$ bimodules $(E_1, \varphi _1 , F_1)$ and $(E_2,\varphi _2, F_2)$ are \textit{homotopic} if there is a Kasparov $G$-equivariant $A$-$IB$ bimodule $(E,\varphi ,F)$ such that $(\ev _i) _*(E,\varphi ,F)$ are unitary equivalent to $(E_i , \varphi _i, F_i)$. 

\begin{defn}
Let $A$ and $B$ be $\sigma$-unital $\zahl /2$-graded $\sigma$-$G$-$\Cst$-algebras. The $\KK$-group $\KK^G(A,B)$ is the set of homotopy equivalence classes of $G$-equivariant Kasparov $A$-$B$ bimodules. 
\end{defn}
It immediately follows from the definition that $\KK^G(\comp, A)$ is canonically isomorphic to the representable equivariant $\K$-group $\mathcal{R}K_0^G(A)$ of Phillips\cite{MR1050490}.

\begin{defn}
Let $(E_1,\varphi_1,F_1)$ be a Kasparov $A$-$B$ $G$-bimodule and $(E_2,\varphi _2 ,F_2)$ a $G$-equivariant Kasparov $B$-$C$ bimodule. A Kasparov product of $(E_1,\varphi_1,F_1)$ and $(E_2,\varphi _2 ,F_2)$ is a $G$-equivariant Kasparov $A$-$C$ bimodule $(E_1 \otimes _B E_2, \varphi , F)$ that satisfies the following.
\begin{enumerate}
\item The operator $F \in \mathbb{L}(E_1 \otimes_B E_2)$ is an \textit{$F_2$-connection}. That is, $T_x \circ F_2 - (-1)^{\deg x \cdot \deg F_2}F \circ T_x$ and $F_2 \circ T_x^* - (-1)^{\deg x \cdot \deg F_2}T_x ^* \circ F$ are compact for any $x \in E_1$.
\item $\varphi (a)[F_1 \otimes 1,F] \varphi (a)^* \geq 0$ mod $\Kop(E)$.
\end{enumerate}
\end{defn}

\begin{thm}
Let $A$, $B$, $C$ and $D$ be $\sigma$-unital $\sigma$-$G$-$\Cst$-algebras. Moreover we assume that $A$ is separable. The Kasparov product gives a well-defined group homomorphism
$$\KK^G(A,B) \otimes \KK^G(B,C) \to \KK^G(A,C)$$
which is associative, that is, $(x \otimes _B y) \otimes _C z=x \otimes _B (y \otimes _C z)$ for any $x \in \KK^G(A,B)$, $y \in \KK^G(B,C)$ and $z \in \KK^G(C,D)$ when $B$ is also separable.
\end{thm}
\begin{proof}
What we have to show is existence, uniqueness up to homotopy, well-definedness of maps between $\KK$-groups and associativity of the Kasparov product. All of them are proved in the same way as in Theorem 12 and Theorem 21 of \cite{MR743845} or Theorem 2.11 and Theorem 2.14 of \cite{MR918241}. Note that we can apply the Kasparov technical lemma \ref{thm:tech} since we may assume that $\sup_{p \in \mathscr{S}(\Lop (E))} p(F) \leq 1$ by a functional calculus and a separable $\sigma$-$\Cst$-algebra is separable as a topological algebra (see also Section 18.3 - 18.6 of \cite{MR1656031}).
\end{proof}

Moreover, we obtain the Puppe exact sequence (as Theorem 19.4.3 of \cite{MR1656031}) for a $\ast$-homomorphism between $\sigma$-$\Cst$-algebras and the six term exact sequences (Theorem 19.5.7 of \cite{MR1656031}) for a semisplit exact sequence of $\sigma$-$\Cst$-algebras by the same proofs. 

Next we deal with the Cuntz picture~\cite{MR733641} (see also \cite{MR1803228}) of $\KK$-theory for $\sigma$-$G$-$\Cst$-algebras.

\begin{defn}[Definition 2.2 of \cite{MR733641}]
We say that $(\varphi _0,\varphi _1) : A \rightrightarrows D \rhd J \to B$ is an equivariant \textit{prequasihomomorphism} from $A$ to $B$ if $D$ is a $\sigma$-unital $\sigma $-$\Cst$-algebra with $G$-action, $\varphi _0$ and $\varphi _1$ are equivariant $\ast$-homomorphisms from $A$ to $D$ such that $\varphi_0 (a)-\varphi _1(a)$ are in a separable $G$-invariant ideal $J$ of $D$ such that the restriction of the $G$-action on $J$ is continuous, and $J \to B$ is an equivariant $\ast$-homomorphism. Moreover we say that $(\varphi_0,\varphi_1)$ is \textit{quasihomomorphism} if $D$ is generated by $\varphi _0(A)$ and $\varphi _1(A)$, $J$ is generated by $\{ \varphi _0(a) -\varphi _1(a) \mid a \in A \}$ and $J \to B$ is injective. 
\end{defn}

The idea given in \cite{MR899916} is also generalized for $\sigma$-$G$-$\Cst$-algebras. 
\begin{defn}
Let $A$ and $B$ be $\sigma$-$G$-$\Cst$-algebras. The full free product $A \ast B$ is the $\sigma$-$G$-$\Cst$-algebra given by the completion of the algebraic free product $A \ast _{\mathrm{alg}} B$ by seminorms
$$ p_{\pi_A,\pi _B}(a_1b_1 \ldots a_nb_n)=\ssbk{\pi _A (a_1)\pi _B (b_1) \ldots \pi _A(a_n)\pi _B (b_n)}$$
where $\pi_A$ and $\pi _B$ are $\ast$-representations of $A$ and $B$ on the same Hilbert space. In other words, when $A=\varprojlim A_n$ and $B=\varprojlim B_m$, the free product $A \ast B$ is the projective limit $\varprojlim (A_n \ast B_m)$.
\end{defn}

By definition, any $\ast$-homomorphisms $\varphi _A :A \to D$ and $\varphi _B : B \to D$ are uniquely extended to $\varphi _A * \varphi _B : A \ast B \to D$. We denote by $QA$ the free product $A \ast A$ and by $qA$ the kernel of the $\ast$-homomorphism $\id _A \ast \id _A : QA \to A$. 

Since we have the stabilization theorem \ref{thm:stab} and the technical theorem \ref{thm:tech} for $\sigma$-$G$-$\Cst$-algebras, the following properties of quasihomomorphisms and $\KK$-theory is proved in the same way. We only enumerate their statements and references for the proofs. Here we write $q_sA$ for the $G$-$\Cst$-algebra $q(A \otimes \Kop _G)$.
\begin{itemize}
\item The set of homotopy classes of $G$-equivariant quasihomomorphisms from $A \otimes \Kop _G$ to $B \otimes \Kop _G$ is isomorphic to $\KK^G(A,B)$ (Section 5 of \cite{MR733641}).
\item The functor $\KK^G: \GCsep{G} \times \GCsep{G}  \to R(G){\mathchar`-}\mathbf{Mod}$ is stable and split exact in both variables (Proposition 2.1 of \cite{MR899916}).
\item For any $\sigma$-$G$-$\Cst$-algebras $A$ and $B$, $A \ast B$ and $A \oplus B$ are $\KK ^G$-equivalent (proof of Proposition 3.1 of \cite{MR899916}).
\item The element $\pi _A:=[\pi _0]$ in $\KK^G(qA,A)$ where $\pi _0:=(\id _A \ast 0 )|_{qA} : qA \to A$ is the $\KK^G$-equivalence (Proposition 3.1 of \cite{MR899916}).
\item There is a one-to-one correspondence between $G$-equivariant quasihomomorphisms from $A \otimes \Kop _G$ to $B \otimes \Kop _G$ and $G$-equivariant $\ast$-homomorphisms from $q_sA$ to $B \otimes \Kop _G$ (Theorem 5.5 of \cite{MR1803228}).
\item There is a canonical isomorphism $\KK^G(A,B) \cong [q_sA,B \otimes \Kop_G]^G$ (the stabilization theorem \ref{thm:stab} and Proposition 1.1 of \cite{MR899916}). 
\item The correspondence
\ma{[q_sA \otimes \Kop _G , q_sB  \otimes \Kop _G]^G &\to \KK^G(A,B)\\
\varphi &\mapsto \pi_B \circ \varphi \circ (\pi _A)^{-1}
}
induces the natural isomorphism (Theorem 6.5 of \cite{MR1803228}).
\end{itemize}

For a projective system $\{ A_n, \pi _n \}$ of $\sigma$-$\Cst$-algebras, the homotopy projective limit $\hoprojlim A_n $ is actually isomorphic to the mapping telescope 
$$\Tel A_n:= \{ f \in \prod C([0,1],A_n) \mid f_n(1)=\pi _n (f(0)) \}.$$
The following theorem follows from the fact that the functor $\KK^G(A,\blank)$ and $\KK^G(\blank , B)$ is compatible with direct products when $B$ is a $G$-$\Cst$-algebra. 

\begin{thm}\label{thm:limit}
The following holds:
\begin{enumerate}
\item Let $\{ A_n \}_{n \in \zahl _{>0}}$ be a inductive system of $\sigma$-$G$-$\Cst$-algebras and $A:=\hoinjlim A_n$. For a $\sigma$-$G$-$\Cst$-algebra $B$, there is an exact sequence
$$0 \to {\varprojlim }^1 \KK ^G_{*+1}(A_n,B) \to \KK^G(A,B) \to \KK^G_*(A_n,B) \to 0.$$
\item Let $\{ B _n \} _{n \in \zahl _{>0}}$ be a projective system of $\sigma$-$G$-$\Cst$-algebras and $B:=\hoprojlim B_n$. For a $\sigma$-$G$-$\Cst$-algebra $B$, there is an exact sequence
$$0 \to {\varprojlim }^1 \KK ^G_{*+1}(A,B_n) \to \KK^G(A,B) \to \varprojlim \KK^G_*(A,B_n) \to 0$$
\item Let $\{ A_n \}_{n \in \zahl _{>0}}$ be a projective system of $\sigma$-$G$-$\Cst$-algebras and $A:=\hoprojlim A_n$. For a $G$-$\Cst$-algebra $B$, there is an isomorphism
$$\KK^G(A,B) \cong \varinjlim \KK^G(A_n,B).$$
\end{enumerate}
\end{thm}

\begin{cor}\label{cor:sigma}
Let $A=\hoprojlim A_n$ and $B=\hoprojlim B_m$ be homotopy projective limits of $\Cst$-algebras. There is an exact sequence
$$0 \to {\varprojlim _m}^1 \varinjlim_n \KK^G_{*+1}(A_n,B_m) \to \KK^G_*(A,B) \to \varprojlim_m \varinjlim_n \KK^G(A_n,B_m) \to 0.$$
\end{cor}
In particular, if two homotopy projective limits $A=\hoprojlim A_n$ and $B=\hoprojlim B_m$ of $G$-$\Cst$-algebras are $\KK^G$-equivalent, then we get a pro-isomorphism of projecctive systems $\{ A_n \}_n \to \{ B_m \}_m$ in $\Kas ^G$.

\subsection{The Kasparov category}
\begin{defn}
We write $\sigma \Kas^G$ for the Kasparov category of $\sigma$-$G$-$\Cst$-algebras i.e.\ the additive category whose objects are separable $\sigma$-$G$-$\Cst$-algebras, morphisms from $A$ to $B$ are $\KK^G(A,B)$ and composition is given by the Kasparov product. 
\end{defn}

Note that the inclusion $\GCsep{G} \subset \sigma \GCsep{G}$ induces that $\Kas ^G$ is a full subcategory of $\sigma \Kas ^G$. Additional structures of $\Kas ^G$ such as tensor products, crossed products and countable direct sums are extended on $\sigma \Kas ^G$. Moreover the category $\Kas ^G$ has countably infinite direct products.

\begin{thm}[Theorem 2.2 of \cite{MR1656818}, Satz 3.5.10 of \cite{Bonkat}]\label{thm:univ}
The category $\sigma \Kas^G$ is an additive category that has the following universal property: there is the canonical functor $\KK^G:\sigma \mathfrak{C}^*\mathfrak{sep} \to \sigma \Kas^G$ such that for any $\Cst$-stable, split-exact, and homotopy invariant functor $F:\sigma \mathfrak{C}^*\mathfrak{sep} \to \mathfrak{A}$ there is a unique functor $\tilde{F}$ such that the following diagram commutes.
\[ 
\xymatrix{
\sigma \mathfrak{C}^*\mathfrak{sep}^G \ar[r] \ar[d] & \mathfrak{A}\\
\sigma \Kas^G \ar@{.>}[ru]}
\]
\end{thm}
This follows from the Cuntz picture introduced in the previous subsection. 

A structure of the triangulated category on $\Kas ^G$ is introduced in \cite{MR2193334}. Let $S$ be the suspension functor $SA:=C_0(\real) \otimes A$ of $\Cst$-algebras. Roughly speaking, the inverse $\Sigma :=S^{-1}$ and the mapping cone exact sequence 
$$\Sigma B \to \mathrm{cone} (f) \to A \xra{f} B$$
determines a triangulated category structure of $\Kas ^G$. More precisely we need to replace the category $\Kas ^G$ with another one that is equivalent to $\Kas ^G$, whose objects are pair $(A,n)$ where $A$ is a separable $\sigma$-$G$-$\Cst$-algebra and $n \in \zahl$, morphisms from $(A,n)$ to $(B, m)$ are $\KK_{n-m}(A,B)$ and composition is given by the Kasparov product. In this category the functor $\Sigma : (A,n) \mapsto (A,n+1)$ is an category isomorphism (not only an equivalence) and $S \circ \Sigma =\Sigma \circ S$ are natural equivalent with the identity functor. A triangle $\Sigma (B,m) \to (C,l) \to (A,n) \to (B,m)$ is exact if there is a $\ast$-homomorphism from $A'$ to $B'$ and the isomorphism $\alpha$, $\beta$ and $\gamma$ such that the following diagram commute. 
\[ 
\xymatrix{
\Sigma B \ar[r] \ar[d]^{\Sigma \beta}_\cong & C \ar[r] \ar[d]^\gamma_\cong & A \ar[r] \ar[d] ^\alpha_\cong & B \ar[d]^\beta_\cong \\
\Sigma B' \ar[r] & \mathrm{cone} (f) \ar[r] & A' \ar[r]^f & B'.
}
\]
For simplicity of notation we use the same letter $\Kas ^G$ for this category.

\begin{thm}\label{thm:triang}
The category $\sigma \Kas^G$, with the suspension $\Sigma$ and exact triangles as above, is a triangulated category.
\end{thm}
We omit the proof. Actually, the same proof as for $\Kas^G$ given in Appendix 1 of \cite{MR2193334} works since we have the Cuntz picture of equivariant $\KK$-theory introduced in the previous subsection. 

\subsection*{Acknowledgement}
The authors would like to thank their supervisor Yasuyuki Kawahigashi for his support and encouragement. They also would like to thank Georges Skandalis and N. Christopher Phillips for suggesting initial ideas and Kang Li for helpful comments.  
The first author wishes to thank the KU Leuven, where the last part of this paper was written, for the invitation and hospitality. This work was supported by Research Fellow of the JSPS (No.\ 26-2598 and No.\ 26-7081) and the Program for Leading Graduate Schools, MEXT, Japan.

\bibliographystyle{alpha}
\bibliography{bibABC,bibDEFG,bibHIJK,bibLMN,bibOPQR,bibSTUV,bibWXYZ,arXiv,AtiyahSegal}
\end{document}